\documentclass{article}
\usepackage[utf8]{inputenc}
\usepackage[utf8]{inputenc}
\usepackage[utf8]{inputenc}
\usepackage{dirtytalk}
\usepackage{tikz}
\usepackage{geometry}
\usepackage{graphicx}
\usepackage{enumitem}
\usepackage{parskip}
\usepackage{amsfonts}
\usepackage{amssymb}
\usepackage{amsmath}
\usepackage{amsthm}
\usepackage{xcolor}
\usepackage{hyperref}
\newlist{myitemize}{itemize}{1}
\setlist[myitemize,1]{leftmargin = 0.5in}
\hypersetup{colorlinks = true, linkcolor = red,  linktocpage}

\theoremstyle{plain}
\newtheorem{thm}{Theorem}[section]
\newtheorem*{thm*}{Theorem}

\newtheorem{cor}[thm]{Corollary}

\theoremstyle{definition}

\newtheorem{conj}[thm]{Conjecture}

\newtheorem{rem}[thm]{Remark}

\title{\textbf{\small{COUNTING THE NUMBER OF $\mathcal{O}_{K}$-FIXED POINTS OF A DISCRETE DYNAMICAL SYSTEM WITH APPLICATIONS FROM ARITHMETIC STATISTICS, II}}}
\author{\footnotesize{BRIAN KINTU}}
\date{\small{\textit{To: Teachers who help to spark a true spirit of learning and discovering in every child in the world}}}

\begin{document}
\maketitle
\begin{abstract}
\small{In this follow-up paper, we again inspect a surprising connection between the set of fixed points of a polynomial map $\varphi_{d,c}$ defined by $\varphi_{d,c}(z) = z^d + c$ for all $c, z \in \mathcal{O}_{K}$ and the coefficient $c$, where $K$ is any number field of degree $n > 1$ and $d > 2$ is an integer. As in \cite{BK1}, we wish to study counting problems which are inspired by exciting advances in arithmetic statistics, and again partly by point-counting result of Narkiewicz on real $K$-rational periodic points of any odd degree map $\varphi_{d,c}$ in arithmetic dynamics. In doing so, we then first prove that for any real algebraic number field $K$ of  degree $n \geq 2$, and for any prime $p \geq 3$ and integer $\ell \geq 1$, the average number of distinct integral fixed points of any $\varphi_{p^{\ell},c}$ modulo prime ideal $p\mathcal{O}_{K}$ is $3$ or $0$ as $c\to \infty$. Motivated further by $K$-rational periodic point-counting result of Benedetto on any $\varphi_{(p-1)^{\ell},c}$ for any prime $p \geq 5$ and for any integer $\ell \in \mathbb{Z}_{\geq 1}$ in arithmetic dynamics, we then also prove unconditionally that for any number field (not necessarily real) $K$ of degree $n \geq 2$, the average number of distinct integral fixed points of any $\varphi_{(p-1)^{\ell},c}$ modulo prime $p\mathcal{O}_{K}$ is $1$ or $2$ or $0$ as $c\to \infty$. Finally, we then apply density result of Bhargava-Shankar-Wang and number field-counting result of Oliver Lemke-Thorne from arithmetic statistics, and as a result obtain again counting and statistical results on the irreducible monic integer polynomials and number fields arising naturally in our polynomial discrete dynamical settings.}
\end{abstract}

\begin{center}
\tableofcontents
\end{center}
\begin{center}
    \section{Introduction} \label{sec1}
\end{center}
\noindent
Consider any morphism $\varphi: {\mathbb{P}^N(K)} \rightarrow {\mathbb{P}^N(K)} $ of degree $d \geq 2$ defined on a projective space ${\mathbb{P}^N(K)}$ of dimension $N$, where $K$ is a number field. Then for any $n\in\mathbb{Z}$ and $\alpha\in\mathbb{P}^N(K)$, we call $\varphi^n = \underbrace{\varphi \circ \varphi \circ \cdots \circ \varphi}_\text{$n$ times}$ the $n^{th}$ \textit{iterate of $\varphi$} and call $\varphi^n(\alpha)$ the \textit{$n^{th}$ iteration of $\varphi$ on $\alpha$}. By convention, $\varphi^{0}$ acts as the identity map, i.e., $\varphi^{0}(\alpha) = \alpha$ for every point $\alpha\in {\mathbb{P}^N(K)}$. The everyday philosopher may want to know (quoting here Devaney \cite{Dev}): \say{\textit{Where do points $\alpha, \varphi(\alpha), \varphi^2(\alpha), \ \cdots\ ,\varphi^n(\alpha)$ go as $n$ becomes large, and what do they do when they get there?}} So now, for any given integer $n\geq 0$ and any given point $\alpha\in {\mathbb{P}^N(K)}$, we then call the set consisting of all the iterates $\varphi^n(\alpha)$ the \textit{(forward) orbit of $\alpha$}; and which in the theory of dynamical systems we usually denote  it by $\mathcal{O}^{+}(\alpha)$.

As we mentioned in the previous work \cite{BK1} that one of the main 
goals in \say{arithmetic dynamics}, is to classify all the points $\alpha\in\mathbb{P}^N(K)$ according to the behavior of their forward orbits $\mathcal{O}^{+}(\alpha)$. In this direction, we recall that any point $\alpha\in {\mathbb{P}^N(K)}$ is called a \textit{periodic point of $\varphi$}, whenever $\varphi^n (\alpha) = \alpha$ for some integer $n\in \mathbb{Z}_{\geq 0}$. In this case, any integer $n\geq 0$ such that the iterate $\varphi^n (\alpha) = \alpha$, is called \textit{period of $\alpha$}; and the smallest such positive integer $n\geq 1$ is called the \textit{exact period of $\alpha$}. We recall Per$(\varphi, {\mathbb{P}^N(K)})$ to denote set of all periodic points of $\varphi$; and also recall that for any given point $\alpha\in$Per$(\varphi, {\mathbb{P}^N(K)})$ the set of all iterates of $\varphi$ on $\alpha$ is called \textit{periodic orbit of $\alpha$}. In their 1994 paper \cite{Russo} and in his 1998 paper \cite{Poonen} respectively, Walde-Russo and Poonen give independently interesting examples of rational periodic points of any $\varphi_{2,c}$ defined over the field $\mathbb{Q}$; and so the interested reader may wish to revisit \cite{Russo, Poonen} to gain familiarity with the notion of periodicity of points.

Previously in article \cite{BK1}, we (inspired greatly by work of Bhargava-Shankar-Tsimerman (BST) and their collaborators in arithmetic statistics, along with work of Narkiewicz \cite{Narkie} in arithmetic dynamics) proved that conditioning on a rational periodic point-counting theorem of Narkiewicz \cite{Narkie} yields that the number of distinct integral fixed points of any $\varphi_{p,c}$ modulo $p$ is equal to $3$ or $0$; from which it then followed that the average number of distinct integral fixed points of any $\varphi_{p,c}$ modulo $p$ is also equal to $3$ or $0$ as $c\to \infty$. Moreover, we then also observed [\cite{BK1}, Remark 2.3] that the expected total number of distinct integral points in the whole family of polynomial maps $\varphi_{p,c}$ modulo $p$ is equal to $3+0 = 3$. So now, inspired again by work of (BST) in arithmetic statistics and also again by work of Narkiewicz \cite{Narkie} in arithmetic dynamics, we then revisit that setting in \cite{BK1} and then condition on an immediate (because $\mathbb{R}$ is an ordered field) $K$-rational periodic point-counting result of Narkiewicz \cite{Narkie} on polynomial maps $\varphi_{p,c}$ defined over a real algebraic number field $K$ of any degree $n \geq 2$. In doing so, we then prove the following main theorem, which we state later more precisely as Theorem \ref{2.2} and which by the same counting argument we then generalize further as Theorem \ref{2.3}; and from which when we restrict our attention on $\mathbb{Z} \subset \mathcal{O}_{K}$ , we then obtain a further generalization Corollary \ref{cor2.4} of [\cite{BK1}, Theorem 2.2]:

\begin{thm}\label{BB} 
Let $K\slash \mathbb{Q}$ be any real number field of degree $ n \geq 2$ with the ring of integers $\mathcal{O}_{K}$, and in which any fixed prime $p\geq 3$ is inert. Assume Thm \ref{theorem 3.2.1}, and let $\varphi_{p, c}$ be a polynomial map defined by $\varphi_{p, c}(z) = z^p + c$ for all $c, z\in\mathcal{O}_{K}$. Then the number of distinct integral fixed points of any $\varphi_{p,c}$ modulo $p\mathcal{O}_{K}$ is either $3$ or zero. 
\end{thm}

Recall further in \cite{BK1} we (again inspired by activity and down-to-earth work of (BST) in arithmetic statistics, and also by a Theorem \ref{main} of Benedetto along with a Conjecture \ref{conjecture 3.2.1} of Hutz and Panraksa's work in arithmetic dynamics) proved that the number of distinct integral fixed points of any $\varphi_{p-1,c}$ modulo $p$ is equal to $1$ or $2$ or $0$; from which it then followed that the average number of distinct integral fixed points of any $\varphi_{p-1,c}$ modulo $p$ is also $1$ or $2$ or $0$ as $c\to \infty$. Moreover, we then also observed [\cite{BK1}, Remark 3.3] that the expected total number of distinct integral points in the whole family of maps $\varphi_{p-1,c}$ modulo $p$ is equal to $1+2+0 = 3$. So now, inspired again by work of (BST) in arithmetic statistics, and by Theorem \ref{main} of Benedetto along with Conjecture \ref{conjecture 3.2.1} of Hutz  (though not more importantly again attempting to prove \ref{conjecture 3.2.1}) and Panraksa's work \cite{par2} in arithmetic dynamics, we revisit the setting in Section \ref{sec2} and then consider in Section \ref{sec3} any $\varphi_{(p-1)^{\ell},c}$ iterated on the space $\mathcal{O}_{K}\slash p\mathcal{O}_{K}$ where $K$ is any number field of degree $n\geq 2$ and $p\geq 5$ is any prime and $\ell \geq 1$ is any integer. In doing so, we then also prove unconditionally the following main theorem on any $\varphi_{p-1,c}$, which we state later more precisely as Theorem \ref{3.2} and then also generalize further as Theorem \ref{3.3}; and which when we again restrict on $\mathbb{Z} \subset \mathcal{O}_{K}$ , we then also obtain a further generalization Corollary \ref{cor3.4} of [\cite{BK1}, Theorem 3.2]:

\begin{thm}\label{Binder-Brian}
Let $K\slash \mathbb{Q}$ be any algebraic number field of degree $n\geq 2$ with the ring of integers $\mathcal{O}_{K}$, and in which any fixed prime $p\geq 5$ is inert. Let $\varphi_{p-1, c}$ be a polynomial map defined by $\varphi_{p-1, c}(z) = z^{p-1} + c$ for all points $c, z\in\mathcal{O}_{K}$. Then the number of distinct integral fixed points of any $\varphi_{p-1,c}$ modulo $p\mathcal{O}_{K}$ is $1$ or $2$ or zero.
\end{thm}

\noindent Notice that the count obtained in Theorem \ref{Binder-Brian} on the number of distinct integral fixed points of any $\varphi_{p-1,c}$ modulo $p\mathcal{O}_{K}$ is independent of $p$ (and so independent of the degree of $\varphi_{p-1,c})$ and $n=[K:\mathbb{Q}]$ in each of the possibilities. Moreover, we may also observe that the expected total count (namely, $1+2+0 =3$) in Theorem \ref{Binder-Brian} on the number of distinct integral fixed points in the whole family of maps $\varphi_{p-1,c}$ modulo $p\mathcal{O}_{K}$ is not only also independent of $p$ (and so independent of deg$(\varphi_{p-1,c})$) and $n$, but is also a constant $3$ as $p-1\to \infty$ or $n\to \infty$; an observation that somewhat surprising coincides not only with a similar observation on the count in each of the possibilities in Theorem \ref{BB} but also coincides with a similar observation on the expected total count (namely, $3+0 =3$) on the number of distinct fixed points in the whole family of maps $\varphi_{p,c}$ modulo $p\mathcal{O}_{K}$.

Since we know the inclusion $\mathbb{Z}\hookrightarrow\mathbb{Z}_{p}$ of rings, and so the space $\mathbb{Z}_{p}$ of all $p$-adic integers is evidently a much larger space than $\mathbb{Z}$. So then, inspired by the work of Adam-Fares \cite{Ada} in arithmetic dynamics and again by a \say{counting-application} philosophy in arithmetic statistics, we revisit in a forthcoming paper \cite{BK3} the setting in Section \ref{sec2} and \ref{sec3} where we do consider $\mathbb{Z}_{p}$, and moreover again with the sole purpose of inspecting the aforementioned relationship. Interestingly, we again obtain the same counting and asymptotics in the setting when $\varphi_{(p-1)^{\ell},c}$ is iterated on the space $\mathbb{Z}_{p}\slash p\mathbb{Z}_{p}$ for any integer $\ell\geq 1$, and obtain very different counting and asymptotics in the case when $\varphi_{p^{\ell},c}$ is iterated on the space $\mathbb{Z}_{p}\slash p\mathbb{Z}_{p}$ for any integer $\ell \geq 1$. Furthermore, motivated further by an $\mathbb{F}_{p}(t)$-periodic point-counting theorem of Benedetto \ref{main} and by a setting in the work of Narkiewicz \cite{Narkie1} on $K$-rational periodic point-counting on maps $\varphi_{p^{\ell},c}$ defined over any totally complex algebraic number field $K$, we revisit the setting in Section \ref{sec2} (though this time allowing $K$ to be any algebraic number field and also more importantly neither condition on Theorem \ref{main} nor on results in \cite{Narkie1}) and Sect. \ref{sec3} in that same work \cite{BK3}. In doing so, we prove in \cite{BK3} that one can obtain a counting and asymptotics in the setting when $\varphi_{(p-1)^{\ell},c}$ is iterated on $\mathbb{F}_{p}[t]\slash (\pi)$ for any fixed irreducible monic $\pi \in \mathbb{F}_{p}[t]$, which is very similar to the one that has been carried out here in Sect. \ref{sec3}; and also prove in \cite{BK3} that one can obtain a counting and asymptotics which is not only the same in the setting when $\varphi_{p^{\ell},c}$ is iterated independently on the space $\mathcal{O}_{K}\slash p\mathcal{O}_{K}$ and $\mathbb{F}_{p}[t]\slash (\pi)$ for any fixed irreducible monic $\pi \in \mathbb{F}_{p}[t]$, but is also  different from the one that has been carried out here in Section \ref{sec2}. 

In addition, to the notion of a periodic point and a periodic orbit, we also recall that a point $\alpha\in {\mathbb{P}^N(K)}$ is called a \textit{preperiodic point of $\varphi$}, whenever the iterate $\varphi^{m+n}(\alpha) = \varphi^{m}(\alpha)$ for some integers $m\geq 0$ and $n\geq 1$. In this case, we recall that the smallest integers $m\geq 0$ and $n\geq 1$ such that $\varphi^{m+n}(\alpha) = \varphi^{m}(\alpha)$ happens, are called the \textit{preperiod} and \textit{eventual period of $\alpha$}, respectively. Again, we denote the set of preperiodic points of $\varphi$ by PrePer$(\varphi, {\mathbb{P}^N(K)})$. For any given preperiodic point $\alpha$ of $\varphi$, we then call the set of all iterates of $\varphi$ on $\alpha$, \textit{the preperiodic orbit of $\alpha$}.
Now observe for $m=0$, we have $\varphi^{n}(\alpha) = \alpha $ and so $\alpha$ is a periodic point of period $n$. Thus, the set  Per$(\varphi, {\mathbb{P}^N(K)}) \subseteq$ PrePer$(\varphi, {\mathbb{P}^N(K)})$; however, it need not be PrePer$(\varphi, {\mathbb{P}^N(K)})\subseteq$ Per$(\varphi, {\mathbb{P}^N(K)})$. In their 2014 paper \cite{Doyle}, Doyle-Faber-Krumm give nice examples (which also recovers examples in Poonen's paper \cite{Poonen}) of preperiodic points of any quadratic map $\varphi$ (where $\varphi$ is not necessarily the somewhat mostly studied $\varphi_{2,c}$ in arithmetic dynamics) defined over quadratic fields; and so the interested reader may wish to revisit \cite{Poonen, Doyle}.

In the year 1950, Northcott \cite{North} used the theory of height functions to show that not only is the set PrePer$(\varphi, {\mathbb{P}^N(K)})$ always finite, but also for a given morphism $\varphi$ the set PrePer$(\varphi, {\mathbb{P}^N(K)})$ can be computed effectively. Forty-five years later, in the year 1995, Morton and Silverman conjectured that PrePer$(\varphi, \mathbb{P}^N(K))$ can be bounded in terms of degree $d$ of $\varphi$, degree $D$ of $K$, and dimension $N$ of the space ${\mathbb{P}^N(K)}$. This celebrated conjecture is called the \textit{Uniform Boundedness Conjecture}; which we then restate here as the following conjecture:

\begin{conj} \label{silver-morton}[\cite{Morton}]
Fix integers $D \geq 1$, $N \geq 1$, and $d \geq 2$. There exists a constant $C'= C'(D, N, d)$ such that for all number fields $K/{\mathbb{Q}}$ of degree at most $D$, and all morphisms $\varphi: {\mathbb{P}^N}(K) \rightarrow {\mathbb{P}^N}(K)$ of degree $d$ defined over $K$, the total number of preperiodic points of a morphism $\varphi$ is at most $C'$, i.e., \#PrePer$(\varphi, \mathbb{P}^N(K)) \leq C'$.
\end{conj}
\noindent A special case of Conjecture \ref{silver-morton} is when the degree $D$ of a number field $K$ is $D = 1$, dimension $N$ of a space $\mathbb{P}^N(K)$ is $N = 1$, and degree $d$ of a morphism $\varphi$ is $d = 2$. In this case, if $\varphi$ is a polynomial morphism, then it is a quadratic map defined over the field $\mathbb{Q}$. Moreover, in this very special case, in the year 1995, Flynn and Poonen and Schaefer conjectured that a quadratic map has no points $z\in\mathbb{Q}$ with exact period more than 3. This conjecture of Flynn-Poonen-Schaefer \cite{Flynn} (which has been resolved for cases $n = 4$, $5$ in \cite{mor, Flynn} respectively and conditionally for $n=6$ in \cite{Stoll} is, however, still open for all integers $n\geq 7$ and moreover, which also Hutz-Ingram \cite{Ingram} gave strong computational evidence supporting it) is restated here formally as the following conjecture. Note that in this same special case, rational points of exact period $n\in \{1, 2, 3\}$ were first found in the year 1994 by Russo-Walde \cite{Russo} and also found in the year 1995 by Poonen \cite{Poonen} using a different set of techniques. We now restate the anticipated conjecture of Flynn-Poonen-Schaefer as the following conjecture:
 
\begin{conj} \label{conj:2.4.1}[\cite{Flynn}, Conjecture 2]
If $n \geq 4$, then there is no quadratic polynomial $\varphi_{2,c }(z) = z^2 + c\in \mathbb{Q}[z]$ with a rational point of exact period $n$.
\end{conj}
Now by assuming Conjecture \ref{conj:2.4.1} and also establishing interesting results on rational preperiodic points, in the year 1998, Poonen \cite{Poonen} then concluded that the total number of rational preperiodic points of any quadratic polynomial $\varphi_{2, c}(z)$ over $\mathbb{Q}$ is at most nine. We restate here formally Poonen's result as the following corollary:
\begin{cor}\label{cor2}[\cite{Poonen}, Corollary 1]
If Conjecture \ref{conj:2.4.1} holds, then $\#$PrePer$(\varphi_{2,c}, \mathbb{Q}) \leq 9$,  for all quadratic maps $\varphi_{2, c}$ defined by $\varphi_{2, c}(z) = z^2 + c$ for all points $c, z\in\mathbb{Q}$.
\end{cor}

On still the same note of exact periods and pre(periodic) points, the next natural question that one could ask is whether the aforementioned phenomenon on exact periods and pre(periodic) points has been investigated in some other cases, namely, when $D\geq 2$, $N\geq 1$ and $d\geq 2$. In the case $D = d = 2$ and $N = 1$, then again if $\varphi$ is a polynomial map, then $\varphi$ is a quadratic map defined over a quadratic field $K = \mathbb{Q}(\sqrt{D'})$. In this case, in the years 1900, 1998 and 2006, Netto \cite{Netto}, Morton-Silverman \cite{Morton} and Erkama \cite{Erkama} resp., found independently a parametrization of a point $c$ in the field $\mathbb{C}$ of all complex points which guarantees $\varphi_{2,c}$ to have periodic points of period $M=4$. And moreover when $c\in \mathbb{Q}$, Panraksa \cite{par1} showed that one gets \textit{all} orbits of length $M = 4$ defined over $\mathbb{Q}(\sqrt{D'})$. For $M=5$, Flynn-Poonen-Schaefer \cite{Flynn} found a parametrization of a point $c\in \mathbb{C}$ that yields points of period 5; however, these periodic points are not in $K$, but rather in some other extension of $\mathbb{Q}$. In the same case $D = d = 2$ and $N = 1$, Hutz-Ingram \cite{Ingram} and Doyle-Faber-Krumm \cite{Doyle} did not find in their computational investigations points $c\in K$ for which $\varphi_{2,c}$ defined over $K$ has $K$-rational points of exact period $M = 5$. Note that to say that the above authors didn't find points $c\in K$ for which $\varphi_{2,c}$ has $K$-rational points of exact period $M=5$, is not the same as saying that such points do not exist; since it's possible that the techniques which the authors employed in their computational investigations may have been far from enabling them to decide concretely whether such points exist or not. In fact, as of the present article, we do not know whether $\varphi_{2,c}$ has $K$-rational points of exact period $5$ or not, but surprisingly from \cite{Flynn, Stoll, Ingram, Doyle} we know that for $c=-\frac{71}{48}$ and $D'=33$ the map $\varphi_{2,c}$ defined over $K = \mathbb{Q}(\sqrt{33})$ has $K$-rational points of exact period $M = 6$; and mind you, this is the only example of $K$-rational points of exact period $M=6$ that is currently known of in the whole literature of arithmetic dynamics. For $M>6$, in 2013, Hutz-Ingram [\cite{Ingram}, Prop. 2 and 3] gave strong computational evidence which showed that for any absolute discriminant $D'$ at most 4000 and any $c\in K$ with a certain logarithmic height, the map $\varphi_{2,c}$ defined over any $K$ has no $K$-rational points of exact period greater than 6. Moreover, the same authors \cite{Ingram} also showed that the smallest upper bound on the size of PrePer$(\varphi_{2,c}, K)$ is 15. A year later, in 2014, Doyle-Faber-Krumm \cite{Doyle} also gave computational evidence on 250000 pairs $(K, \varphi_{2,c})$ which not only established the same claim [\cite{Doyle}, Thm 1.2] as that of Hutz-Ingram \cite{Ingram} on the upper bound of the size of PrePer$(\varphi_{2,c}, K)$, but also covered Poonen's claims in \cite{Poonen} on $\varphi_{2,c}$ over $\mathbb{Q}$. Three years later, in 2018, Doyle \cite{Doy} adjusted the computations in his aforementioned cited work with Faber and Krumm; and after from which he made the following conjecture on any quadratic map over any $K = \mathbb{Q}(\sqrt{D'})$:

\begin{conj}\label{do}[\cite{Doy}, Conjecture 1.4]
Let $K\slash \mathbb{Q}$ be a quadratic field and let $f\in K[z]$ be a quadratic polynomial. \newline Then, $\#$PrePer$(f, K)\leq 15$.
\end{conj}

Since Per$(\varphi, {\mathbb{P}^N(K)}) \subseteq$ PrePer$(\varphi, {\mathbb{P}^N(K)})$ and so if the size of PrePer$(\varphi, \mathbb{P}^N(K))$ is bounded above, then so is also the size of Per$(\varphi, \mathbb{P}^N(K))$, and moreover bounded above by the same upper bound. Hence, we may then focus on a periodic version \ref{per} of \ref{silver-morton}, and this is because in Section \ref{sec2} and \ref{sec3} we consider polynomial maps of any odd prime power degree $d\geq 3$ defined over the ring of integers $\mathcal{O}_{K}$ where $K$ is any number field of degree $n\geq 2$ and also consider polynomial maps of even degree $d\geq 4$ defined over $\mathcal{O}_{K}$ where $K$ is any number field of degree $n\geq 2$, respectively, in the attempt of understanding the possibility and validity of such a version of \ref{silver-morton}:

\begin{conj} \label{silver-morton 1}($(D,1)$-version of Conjecture \ref{silver-morton})\label{per}
Fix integers $D \geq 1$ and $d \geq 2$. There exists a constant $C'= C'(D, d)$ such that for all number fields $K/{\mathbb{Q}}$ of degree at most $D$, and all morphisms $\varphi: {\mathbb{P}}^{1}(K) \rightarrow {\mathbb{P}}^{1}(K)$ of degree $d$ over $K$, the total number of periodic points of a morphism $\varphi$ is at most $C'$, i.e., \#Per$(\varphi, \mathbb{P}^{1}(K)) \leq C'$.
\end{conj}

\subsection*{History on the Connection Between the Size of Per$(\varphi_{d, c}, K)$ and the Coefficient $c$}

In the year 1994, Walde and Russo not only proved [\cite{Russo}, Corollary 4] that for a quadratic map $\varphi_{2,c}$ defined over $\mathbb{Q}$ with a periodic point, the denominator of a rational point $c$, denoted as den$(c)$, is a square but they also proved that den$(c)$ is even, whenever $\varphi_{2,c}$ admits a rational cycle of length $\ell \geq 3$. Moreover, Walde-Russo also proved [\cite{Russo}, Cor. 6, Thm 8 and Cor. 7] that the size \#Per$(\varphi_{2, c}, \mathbb{Q})\leq 2$, whenever den$(c)$ is an odd integer. 

Three years later, in the year 1997, Call-Goldstine \cite{Call} proved that the size of PrePer$(\varphi_{2,c},\mathbb{Q})$ can be bounded above in terms of the number of distinct odd primes dividing den$(c)$. We restate formally this result of Call-Goldstine as the following theorem, in which $GCD(a, e)$ refers to the greatest common divisor of $a$, $e \in \mathbb{Z}$:

\begin{thm}\label{2.3.1}[\cite{Call}, Theorem 6.9]
Let $e>0$ be an integer and let $s$ be the number of distinct odd prime factors of e. Define $\varepsilon  = 0$, $1$, $2$, if $4\nmid e$, if $4\mid e$ and $8 \nmid e$, if $8 \mid e$, respectively. Let $c = a/e^2$, where $a\in \mathbb{Z}$ and $GCD(a, e) = 1$. If $c \neq -2$, then the total number of $\mathbb{Q}$-preperiodic points of $\varphi_{2, c}$ is at most $2^{s + 2 + \varepsilon} + 1$. Moreover, a quadratic map $\varphi_{2, -2}$ has exactly six rational preperiodic points.
\end{thm}

Eight years later, after the work of Call-Goldstine, in the year 2005, Benedetto \cite{detto} studied polynomial maps $\varphi$ of arbitrary degree $d\geq 2$ defined over an arbitrary global field $K$, and then established the following result on the relationship between the size of the set PrePre$(\varphi, K)$ and the number of bad primes of $\varphi$ in $K$:

\begin{thm}\label{main} [\cite{detto}, Main Theorem]
Let $K$ be a global field, $\varphi\in K[z]$ be a polynomial of degree $d\geq 2$ and $s$ be the number of bad primes of $\varphi$ in $K$. The number of preperiodic points of $\varphi$ in $\mathbb{P}^N(K)$ is at most $O(\text{s log s})$. 
\end{thm}
 
\noindent Since Benedetto's Theorem \ref{main} applies to any polynomial $\varphi$ of arbitrary degree $d\geq 2$ defined over any number field $K$, it then follows that one can immediately apply Benedetto's Thm \ref{main} to any polynomial $\varphi$ of arbitrary odd or even degree $d> 2$ defined over any number field $K$ and then obtain the upper bound in Theorem \ref{main}. 

Seven years after the work of Benedetto, in the year 2012, Narkiewicz's work \cite{Narkie} also showed that any $\varphi_{d,c}$ of odd degree $d\geq 3$ defined over any real algebraic number field $K\slash\mathbb{Q}$, has no $K$-periodic points of exact period $n \geq 2$; and moreover it also followed from \cite{Narkie} that the total number of rational preperiodic points of any $\varphi_{d,c}$ of odd degree $d\geq 3$ is at most four. Now since the set $\mathbb{R}$ of all real numbers is an ordered field and so is every real number field $K\slash \mathbb{Q}$, then as noted by Narkiewicz \cite{Narkie} that the polynomial $\varphi_{d,c}(x)$ is increasing on $K$ and hence on the subset $K\setminus \mathbb{Q}$ of all the irrational numbers in $K$. So then, one can also prove that Narkiewicz's upper bound $4$ on the number of $\mathbb{Q}$-preperiodic points of any $\varphi_{d,c}$ of odd degree $d\geq 3$, should be the same upper bound on the total number of real $K$-preperiodic points of any $\varphi_{d,c}$ of odd degree $d\geq 3$; and so we have: 

\begin{thm} \label{theorem 3.2.1}\cite{Narkie}
For any real algebraic number field $K\slash \mathbb{Q}$, any integer $n > 1$ and for any odd integer $d\geq 3$, there is no point $c\in K$ such that the polynomial map $\varphi_{d,c}$ defined by $\varphi_{d, c}(z) = z^d + c$ for all $c,z \in K$ has $K$-rational periodic points of exact period $n$. Moreover, $\#$PrePer$(\varphi_{d, c}, \mathbb{Q}) \leq 4$ and so is also $\#$PrePer$(\varphi_{d, c}, K) \leq 4$.
\end{thm} 

\begin{rem}
The first part of Theorem \ref{theorem 3.2.1} is proved by observing that for each odd degree $d\geq 3$, the polynomial $\varphi_{d, c}(z)$ is non-decreasing on $\mathbb{Q}$ and so by elementary mathematical analysis one then expects the forward orbit $\mathcal{O}_{\varphi_{d,c}} (z)$ of each rational point $z$ to form a non-decreasing sequence of iterations. Hence, it is immediately evident that $\varphi_{d, c}(z)$ can only have rational points of exact period $n = 1$ (with no preperiod). Note that over $\mathbb{Q}$ and hence as also mentioned in the paragraph before Theorem \ref{theorem 3.2.1} that over $K$, the upper bound 4 is obtained by counting the number of $K$-rational roots of $\varphi_{d, c}(z)-z = z^d - z + c$. So now, Thm \ref{theorem 3.2.1} shows that total the number of real $K$-points $z$ that satisfy $z^d - z + c = 0$ is bounded above by 4; from which it then follows that the total number of real $\mathcal{O}_{K}$-points
$z$ satisfying $z^d - z + c = 0$ is strictly less than 4 or equal to 4.
\end{rem}

Three years after \cite{Narkie}, in 2015, Hutz \cite{Hutz} developed an algorithm determining effectively all $\mathbb{Q}$-preperiodic points of a morphism defined over a given number field $K$; from which he then made the following conjecture: 

\begin{conj} \label{conjecture 3.2.1}[\cite{Hutz}, Conjecture 1a]
For any integer $n > 2$, there is no even degree $d > 2$ and no point $c \in \mathbb{Q}$ such that the polynomial map $\varphi_{d, c}$ has rational points of exact period $n$.
Moreover, \#PrePer$(\varphi_{d, c}, \mathbb{Q}) \leq 4$. 
\end{conj}

\begin{rem} \label{nt}
If Conjecture \ref{conjecture 3.2.1} held, it would then follow that the total number of integral fixed points of any $\varphi_{d,c}(x)$ of even degree $d>2$ is equal to $4$ or $<4$. Moreover, since the monic polynomial $\varphi_{d,c}(x)\in \mathbb{Z}[x]$ has good reduction modulo $p$, it then follows that the total number of integral fixed points of any $\varphi_{d,c}(x)$ modulo $p$ is also $4$ or $<4$. But of course now the issue is that we unfortunately don't know (as to the author's knowledge) whether Conjecture \ref{conjecture 3.2.1} holds or not, let alone whether $4$ is the correct upper bound on the total number of rational (and hence integral) fixed points of any $\varphi_{d,c}$ of even degree $d>2$. On whether any progress has been made on Conjecture \ref{conjecture 3.2.1}, recently Panraksa \cite{par2} proved among many other results that $\varphi_{4,c}(z)\in\mathbb{Q}[z]$ has $\mathbb{Q}$-points of exact period $n = 2$. Moreover, he also proved that $\varphi_{d,c}(z)\in\mathbb{Q}[z]$ has no $\mathbb{Q}$-points of exact period $n = 2$ for any $c \in \mathbb{Q}$ with $c \neq -1$ and $d = 6$, $2k$ with $3 \mid 2k-1$. The interested reader may find these mentioned results of Panraksa in his unconditional Thms. 2.1, 2.4 and Thm. 1.7 conditioned on the abc-conjecture in \cite{par2}.
\end{rem}

Twenty-eight years later, after the work of Walde-Russo, in the year 2022, Eliahou-Fares proved [\cite{Shalom2}, Theorem 2.12] that the denominator of a rational point $-c$, denoted as den$(-c)$ is divisible by 16, whenever $\varphi_{2,-c}$ defined by $\varphi_{2, -c}(z) = z^2 - c$ for all $c, z\in \mathbb{Q}$ admits a rational cycle of length $\ell \geq 3$. Moreover, they also proved [\cite{Shalom2}, Proposition 2.8] that the size \#Per$(\varphi_{2, -c}, \mathbb{Q})\leq 2$, whenever den$(-c)$ is an odd integer. Motivated by \cite{Call}, Eliahou-Fares \cite{Shalom2} also proved that the size of Per$(\varphi_{2, -c}, \mathbb{Q})$ can be bounded above by using information on den$(-c)$, namely, information in terms of the number of distinct primes dividing den$(-c)$. Moreover, they in \cite{Shalom1} also showed that the upper bound is four, whenever $c\in \mathbb{Q^*} = \mathbb{Q}\setminus\{0\}$. We restate here their results as:

\begin{cor}\label{sha}[\cite{Shalom2, Shalom1}, Cor. 3.11 and Cor. 4.4, respectively]
Let $c\in \mathbb{Q}$ such that den$(c) = d^2$ with $d\in 4 \mathbb{N}$. Let $s$ be the number of distinct primes dividing $d$. Then, the total number of $\mathbb{Q}$-periodic points of $\varphi_{2, -c}$ is at most $2^s + 2$. Moreover, for $c\in \mathbb{Q^*}$ such that the den$(c)$ is a power of a prime number. Then, $\#$Per$(\varphi_{2, c}, \mathbb{Q}) \leq 4$.
\end{cor}

\noindent Once again, the purpose of this article is to inspect further the above connection, however, in the case of polynomial maps $\varphi_{p^{\ell}, c}$ defined over a real algebraic number field $K\slash \mathbb{Q}$ of any degree $n\geq 2$ for any given prime integer $p> 2$, and also in the case of polynomial maps $\varphi_{(p-1)^{\ell}, c}$ defined over any algebraic number field $K\slash \mathbb{Q}$ of degree $n\geq 2$ for any given prime integer $p>3$ and for any integer $\ell\geq 1$; and doing all of this from a sense that (as in \cite{BK1}) is greatly inspired and guided by some of the many striking developments in arithmetic statistics.

\section{On the Number of Integral Fixed Points of any Family of Polynomial Maps $\varphi_{p^{\ell},c}$}\label{sec2}
In this section, we use finite field facts combined with good reduction fact applied on Theorem \ref{theorem 3.2.1} to count the number of fixed points of any $\varphi_{p^{\ell},c}$ modulo prime $p\mathcal{O}_{K}$ for any prime $p\geq 3$ and $\ell \in \mathbb{Z}_{\geq 1}$. To do so, we let $c\in \mathcal{O}_{K}$ be any point, $p\geq 3$ be any prime and $\ell\geq 1$ be any integer, and define fixed point-counting function  
\begin{equation}\label{N_{c}}
N_{c}(p) := \# \bigg\{ z\in \mathcal{O}_{K} / p\mathcal{O}_{K} : \varphi_{p^{\ell},c}(z) - z \equiv 0 \ \text{(mod $p\mathcal{O}_{K}$)}\bigg\}.
\end{equation}\noindent Setting $\ell =1$, so that the map $\varphi_{p^{\ell}, c} = \varphi_{p,c}$, we then first prove the following theorem and its generalization \ref{2.2}:

\begin{thm} \label{2.1}
Let $K\slash \mathbb{Q}$ be a real algebraic number field of any degree $n \geq 2$ with the ring of integers $\mathcal{O}_{K}$, and in which $3$ is inert. Let  $\varphi_{3, c}$ be a cubic map defined by $\varphi_{3, c}(z) = z^3 + c$ for all $c, z\in\mathcal{O}_{K}$, and let $N_{c}(3)$ be as in \textnormal{(\ref{N_{c}})}. Then $N_{c}(3) = 3$ for every coefficient $c\in 3\mathcal{O}_{K}$; otherwise we have $N_{c}(3) = 0$ for every coefficient $c\not \in 3\mathcal{O}_{K}$.
\end{thm}

\begin{proof}
Let $f(z)= \varphi_{3, c}(z)-z = z^3 - z + c$ and note that for every coefficient $c\in 3\mathcal{O}_{K}$, reducing the cubic polynomial $f(z)$ modulo prime ideal $3\mathcal{O}_{K}$, we then obtain $f(z)\equiv z^3 - z$ (mod $3\mathcal{O}_{K}$); and so the reduced cubic polynomial $f(z)$ modulo $3\mathcal{O}_{K}$ factors as $z(z-1)(z+1)$ over a finite field $\mathcal{O}_{K}\slash 3\mathcal{O}_{K}$ of order $3^{[K:\mathbb{Q}]} = 3^n$. But now, it then follows by the factor theorem that $z\equiv -1, 0, 1$ (mod $3\mathcal{O}_{K}$) are roots of $f(z)$ modulo $3\mathcal{O}_{K}$ in $\mathcal{O}_{K}\slash 3\mathcal{O}_{K}$. Moreover, since the univariate polynomial $f(x)$ modulo $3\mathcal{O}_{K}$ is of degree 3 over a field $\mathcal{O}_{K}\slash 3\mathcal{O}_{K}$ and so the reduced polynomial $f(x)$ modulo $3\mathcal{O}_{K}$ can have at most three roots in $\mathcal{O}_{K}\slash 3\mathcal{O}_{K}$(even counted with multiplicity), we then conclude $\#\{ z\in \mathcal{O}_{K} / 3\mathcal{O}_{K} : \varphi_{3,c}(z) - z \equiv 0 \text{ (mod $3\mathcal{O}_{K}$)}\} = 3$ and so the number $N_{c}(3) = 3$. To see 
$N_{c}(3) = 0$ for every coefficient $c\not \in 3\mathcal{O}_{K}$, we first note that since $c\not\in 3\mathcal{O}_{K}$, then this also means $c\not \equiv 0$ (mod $3\mathcal{O}_{K}$). Now since we know from a well-known fact about subfields of finite fields that every subfield of a finite field $\mathcal{O}_{K}\slash 3\mathcal{O}_{K}$ is of order $3^r$ for some positive integer $r\mid n$, we then have the inclusion $\mathbb{F}_{3}\hookrightarrow\mathcal{O}_{K}\slash 3\mathcal{O}_{K}$ of fields, where $\mathbb{F}_{3}$ is a field of order 3; and moreover we also recall (as a fact) that $z^3 = z$ for every element $z\in \mathbb{F}_{3} \subset\mathcal{O}_{K}\slash 3\mathcal{O}_{K}$. But now we note that $z^3 - z + c\not \equiv 0$ (mod $3\mathcal{O}_{K}$) for every element $z\in \mathbb{F}_{3} \subset\mathcal{O}_{K}\slash 3\mathcal{O}_{K}$; and so $f(z)\not \equiv 0$ (mod $3\mathcal{O}_{K}$) for every point $z\in \mathbb{F}_{3} \subset\mathcal{O}_{K}\slash 3\mathcal{O}_{K}$. Otherwise, if $f(\alpha) \equiv 0$ (mod $3\mathcal{O}_{K}$) for some point $\alpha\in \mathcal{O}_{K}\slash 3\mathcal{O}_{K}\setminus \mathbb{F}_{3}$ and for every $c\not\in 3\mathcal{O}_{K}$. Then together with the three roots proved earlier, it follows that the reduced cubic polynomial $f(x)$ modulo $3\mathcal{O}_{K}$ has in total more than three roots in $\mathcal{O}_{K}\slash 3\mathcal{O}_{K}$; which then yields contradiction. This then means $f(x)=\varphi_{3,c}(x)-x$ has no roots in $\mathcal{O}_{K} / 3\mathcal{O}_{K}$ for every coefficient $c\not \in 3\mathcal{O}_{K}$; and so we conclude $N_{c}(3) = 0$. This then completes the whole proof, as desired.
\end{proof} 
We now wish to generalize Theorem \ref{2.1} to every $\varphi_{p, c}$ for any given prime $p\geq 3$. That is, assuming Theorem \ref{theorem 3.2.1}, we prove that the number of distinct fixed points of every $\varphi_{p, c}$ modulo $p\mathcal{O}_{K}$ is also either $3$ or $0$:

\begin{thm} \label{2.2}
Let $K\slash \mathbb{Q}$ be a real number field of any degree $ n \geq 2$ with the ring of integers $\mathcal{O}_{K}$, and in which any fixed prime $p\geq 3$ is inert. Assume Theorem \ref{theorem 3.2.1}, and let $\varphi_{p, c}$ be defined by $\varphi_{p, c}(z) = z^p + c$ for all $c, z\in\mathcal{O}_{K}$ and $N_{c}(p)$ be as in \textnormal{(\ref{N_{c}})}. Then $N_{c}(p) = 3$ for every $c\in p\mathcal{O}_{K}$; otherwise $N_{c}(p) = 0$ for every $c \not \in p\mathcal{O}_{K}$. 
\end{thm}
\begin{proof}
By applying a somewhat similar argument as in Proof of Theorem \ref{2.1}, we then obtain the count as desired. That is, let $f(z)=\varphi_{p, c}(z)-z= z^p - z + c$ and note that for every coefficient $c \in p\mathcal{O}_{K}$, reducing $f(z)$ modulo prime $p\mathcal{O}_{K}$, we then obtain $f(z)\equiv z^p - z$ (mod $p\mathcal{O}_{K}$); and so $f(z)$ modulo $p\mathcal{O}_{K}$ is now a polynomial defined over a finite field $\mathcal{O}_{K}\slash p\mathcal{O}_{K}$ of order $p^n$. So now, since $z$ and $z-1$ are linear factors of $f(z)\equiv z(z-1)(z^{p-2}+ z^{p-3}+ \cdots +1)$ (mod $p\mathcal{O}_{K}$), it then follows by the factor theorem that $f(z)$ (mod $p\mathcal{O}_{K}$) vanishes at $z\equiv 0, 1$ (mod $p\mathcal{O}_{K}$); and so $z\equiv 0, 1$ (mod $p\mathcal{O}_{K}$) are roots of $f(z)$ modulo $p\mathcal{O}_{K}$ in $\mathcal{O}_{K}\slash p\mathcal{O}_{K}$. Now since $K\slash \mathbb{Q}$ is a real algebraic number field, we then by Theorem \ref{theorem 3.2.1} obtain the estimate $\#\{ z\in  \mathcal{O}_{K} : f(z) = 0\}\leq 4$; and moreover since by [\cite{Silverman}, Corollary 2.20] we know that good reduction modulo prime $p\mathcal{O}_{K}$ preserves periodicity of points, then together with the above obtained two roots, it then follows that at most two additional roots $z\in \mathcal{O}_{K}\slash p\mathcal{O}_{K}$ are possible for the reduced monic polynomial $f(x)$ modulo $p\mathcal{O}_{K}$. But now taking also into account of the first part of Theorem \ref{2.1}, we then conclude that for any fixed inert prime $p\geq 3$, the reduced polynomial $f(z)$ (mod $p\mathcal{O}_{K}$) has exactly three distinct roots in $\mathcal{O}_{K} / p\mathcal{O}_{K}$. Hence, we then conclude $\#\{ z\in \mathcal{O}_{K} / p\mathcal{O}_{K} :  \varphi_{p,c}(z) - z \equiv 0 \text{ (mod $p\mathcal{O}_{K}$)}\} = 3$ and so $N_{c}(p) = 3$. To see $N_{c}(p) = 0$ for every coefficient $c \not \in p\mathcal{O}_{K}$, we first note that since $c\not\in p\mathcal{O}_{K}$, then $c\not \equiv 0$ (mod $p\mathcal{O}_{K}$). So now, as before we may recall $\mathbb{F}_{p}\hookrightarrow \mathcal{O}_{K}\slash p\mathcal{O}_{K}$ of fields, where $\mathbb{F}_{p}$ is a field of order $p$; and also recall (as a fact) that $z^p = z$ for every $z\in \mathbb{F}_{p} \subset\mathcal{O}_{K}\slash p\mathcal{O}_{K}$. But now we note $z^p - z + c\not \equiv 0$ (mod $p\mathcal{O}_{K}$) for every element $z\in \mathbb{F}_{p} \subset\mathcal{O}_{K}\slash p\mathcal{O}_{K}$; and so $f(z)\not \equiv 0$ (mod $p\mathcal{O}_{K}$) for every point $z\in \mathbb{F}_{p} \subset\mathcal{O}_{K}\slash p\mathcal{O}_{K}$. Otherwise, if $f(\alpha) \equiv 0$ (mod $p\mathcal{O}_{K}$) for some point $\alpha\in \mathcal{O}_{K}\slash p\mathcal{O}_{K}\setminus \mathbb{F}_{p}$ and for any fixed inert prime $p\geq 3$ and for every $c\not\in p\mathcal{O}_{K}$. Then together with the three roots proved in the first part, it then follows that $f(x)$ modulo $p\mathcal{O}_{K}$ has in total more than three roots in $\mathcal{O}_{K}\slash p\mathcal{O}_{K}$ for any fixed inert prime $p\geq 3$; and from which we then obtain a contradiction, when $p=3$. This then means that $f(x)=\varphi_{p,c}(x)-x$ has no roots in $\mathcal{O}_{K} \slash p\mathcal{O}_{K}$ for every coefficient $c\not \in p\mathcal{O}_{K}$; and so we then conclude $N_{c}(p) = 0$. This then completes the whole proof, as needed.
\end{proof}

Now we wish to generalize Thm \ref{2.2} further to any $\varphi_{p^{\ell}, c}$ for any prime $p\geq 3$ and $\ell \in \mathbb{Z}_{\geq 1}$. That is, assuming Thm \ref{theorem 3.2.1}, we prove that the number of distinct fixed points of any $\varphi_{p^{\ell}, c}$ modulo $p\mathcal{O}_{K}$ is also $3$ or $0$:

\begin{thm} \label{2.3}
Let $K\slash \mathbb{Q}$ be a real number field of any degree $ n \geq 2$ with ring $\mathcal{O}_{K}$, and such that any fixed prime $p\geq 3$ is inert. Assume Theorem \ref{theorem 3.2.1}, and let $\varphi_{p^{\ell}, c}$ be defined by $\varphi_{p^{\ell}, c}(z) = z^{p^{\ell}} + c$ for all $c, z\in\mathcal{O}_{K}$ and $\ell\in \mathbb{Z}_{\geq 1}$. Let $N_{c}(p)$ be as in \textnormal{(\ref{N_{c}})}. Then $N_{c}(p) = 3$ for any point $c\in p\mathcal{O}_{K}$; otherwise $N_{c}(p) = 0$ for any $c \not \in p\mathcal{O}_{K}$. 
\end{thm}
\begin{proof}
By applying a similar argument as in the Proof of Theorem \ref{2.2}, we then obtain the count as desired. That is, let $f(z)=\varphi_{p^{\ell}, c}(z)-z= z^{p^{\ell}} - z + c$ and note that for every coefficient $c \in p\mathcal{O}_{K}$, reducing  $f(z)$ modulo prime  $p\mathcal{O}_{K}$, we then obtain $f(z)\equiv z^{p^{\ell}} - z$ (mod $p\mathcal{O}_{K}$); and so the reduced polynomial $f(z)$ modulo $p\mathcal{O}_{K}$ is now a polynomial defined over a finite field $\mathcal{O}_{K}\slash p\mathcal{O}_{K}$. But now we note $z\equiv 0, 1$ (mod $p\mathcal{O}_{K}$) are roots of $f(z)$ modulo $p\mathcal{O}_{K}$ in $\mathcal{O}_{K}\slash p\mathcal{O}_{K}$, since again $z$ and $z-1$ are linear factors of $f(z)\equiv z(z-1)(z^{p^{\ell}-2}+ z^{p^{\ell}-3}+ \cdots +1)$ (mod $p\mathcal{O}_{K}$). Now since $K\slash \mathbb{Q}$ is a real number field and $\varphi_{p^{\ell},c}$ is of odd degree $p^{\ell}$ for any $\ell\in \mathbb{Z}_{ \geq 1}$, we then by Theorem \ref{theorem 3.2.1} note that $\#\{ z\in  \mathcal{O}_{K} : z^{p^{\ell}} - z + c = 0\}\leq 4$. Moreover, since deg$(f(x)) = p^{\ell} =$ deg$(f(x) \text{ mod } p\mathcal{O}_{K})$ and so $f(x)$ has good reduction modulo $p\mathcal{O}_{K}$ and so (by [\cite{Silverman}, Corollary 2.20]) periodicity of points is preserved, then together with the above two obtained roots, it then follows that at most two additional roots $z\in \mathcal{O}_{K}\slash p\mathcal{O}_{K}$ are possible for $f(x)$ modulo $p\mathcal{O}_{K}$. But now taking also into account of the first part of Theorem \ref{2.2}, we then conclude that for any fixed inert prime $p\geq 3$ and for any $\ell\in \mathbb{Z}_{ \geq 1}$, the number $\#\{ z\in \mathcal{O}_{K} / p\mathcal{O}_{K} :  \varphi_{p^{\ell},c}(z) - z \equiv 0 \text{ (mod $p\mathcal{O}_{K}$)}\} = 3$ and so $N_{c}(p) = 3$. Finally, to see that the number $N_{c}(p) = 0$ for every coefficient $c \not \in p\mathcal{O}_{K}$, let's for the sake of a contradiction, suppose $f(z)=z^{p^{\ell}} - z + c \equiv 0$ (mod $p\mathcal{O}_{K}$) for some point $z\in \mathcal{O}_{K} / p\mathcal{O}_{K}$ and for every $\ell \in \mathbb{Z}_{\geq 1}$. Now note that for $\ell =1$ or $\ell = p$, we then obtain $z^p - z + c \equiv 0$ (mod $p\mathcal{O}_{K}$) or $z^{p^p} - z + c \equiv 0$ (mod $p\mathcal{O}_{K}$) for some $z\in \mathcal{O}_{K} / p\mathcal{O}_{K}$. But now in both case $\ell =1$ and $\ell = p$, we then note that $c\equiv 0$ (mod $p\mathcal{O}_{K}$) since $z^p - z =0$ for every  $z\in \mathbb{F}_{p}\subset\mathcal{O}_{K} / p\mathcal{O}_{K}$; which then yields a contradiction. Otherwise, if $f(\alpha) \equiv 0$ (mod $p\mathcal{O}_{K}$) for some point $\alpha\in \mathcal{O}_{K}\slash p\mathcal{O}_{K}\setminus \mathbb{F}_{p}$ and for any fixed inert prime $p\geq 3$ and for every $\ell \in \mathbb{Z}_{ \geq 1}$ and for every $c\not\in p\mathcal{O}_{K}$. Then note that together with the three roots proved in the first part, then $f(x)$ modulo $p\mathcal{O}_{K}$ has in total more than three roots in $\mathcal{O}_{K}\slash p\mathcal{O}_{K}$ for any fixed inert prime $p\geq 3$ and $\ell \geq 1$; from which we then obtain a contradiction, for any fixed inert prime $p\geq 3$ and $\ell = 1$. This then means $f(x)=\varphi_{p^{\ell},c}(x)-x$ has no roots in $\mathcal{O}_{K} / p\mathcal{O}_{K}$ for every coefficient $c\not \in p\mathcal{O}_{K}$ and so we then conclude $N_{c}(p) = 0$. This completes the whole proof, as needed. 
\end{proof}

Restricting on the subring $\mathbb{Z}\subset \mathcal{O}_{K}$ of ordinary integers, we then also obtain the following consequence of Theorem \ref{2.3} on the number of integral fixed points of any $\varphi_{p^{\ell},c}$ modulo $p$ for any prime $p\geq 3$ and $\ell \in \mathbb{Z}_{\geq 1}$: 

\begin{cor} \label{cor2.4}
Let $p\geq 3$ be any fixed prime integer, and $\ell \geq 1$ be any integer, and assume Theorem \ref{theorem 3.2.1}. Let $\varphi_{p^{\ell}, c}$ be defined by $\varphi_{p^{\ell}, c}(z) = z^{p^{\ell}} + c$ for all $c, z\in\mathbb{Z}$, and let $N_{c}(p)$ be defined as in \textnormal{(\ref{N_{c}})} with $\mathcal{O}_{K} / p\mathcal{O}_{K}$ replaced with $\mathbb{Z}\slash p\mathbb{Z}$. Then $N_{c}(p) = 3$ for every coefficient $c = pt$; otherwise we have $N_{c}(p) = 0$ for any coefficient $c \neq  pt$.
\end{cor}

\begin{proof}
By applying a similar argument as in the Proof of Theorem \ref{2.3}, we then obtain the count as desired
\end{proof}

\begin{rem}\label{re2.5}
With now Theorem \ref{2.3}, we may then to each distinct integral fixed point of $\varphi_{p^{\ell},c}$ associate an integral fixed orbit. In doing so, we then obtain a dynamical translation of Theorem \ref{2.3}, namely, that the number of distinct integral fixed orbits that any $\varphi_{p^{\ell},c}$ has when iterated on the space $\mathcal{O}_{K} / p\mathcal{O}_{K}$ is equal to $3$ or $0$. More to this, we notice that in both of the coefficient cases $c\equiv 0$ (mod $p\mathcal{O}_{K})$ and $c\not \equiv 0$ (mod $p\mathcal{O}_{K})$ considered in Theorem \ref{2.3}, the count obtained on the number of distinct integral fixed points of any $\varphi_{p^{\ell},c}$ modulo $p\mathcal{O}_{K}$ is independent of $p$ (and hence independent of both the degree of $\varphi_{p^{\ell},c}$ for every integer $\ell \in \mathbb{Z}_{\geq 1}$) and the degree $n=[K:\mathbb{Q}]$. Moreover, we may also observe that the expected total count (namely, $3+0 =3$) in Theorem \ref{2.3} on the number of distinct integral fixed points in the whole family of maps $\varphi_{p^{\ell},c}$ modulo $p\mathcal{O}_{K}$ is not only also independent of $p$ (and so independent of deg$(\varphi_{p^{\ell},c})$) and $n$, but is also a constant $3$ even as $p^{\ell}\to \infty$ or $n\to \infty$. 
\end{rem}

\section{The Number of Integral Fixed Points of any Family of Polynomial Maps $\varphi_{(p-1)^{\ell},c}$}\label{sec3}
Unlike in Sect.\ref{sec2} in which we are assumed a theorem, we in this section also wish to unconditionally count the number of fixed points of any $\varphi_{(p-1)^{\ell},c}$ modulo prime  $p\mathcal{O}_{K}$ for any prime $p\geq 5$ and $\ell \in \mathbb{Z}_{\geq 1}$. Again, let $c\in \mathcal{O}_{K}$ be any point, $p\geq 5$ be any prime and $\ell\geq 1$ be any integer, and then define fixed point-counting function  
\begin{equation}\label{M_{c}}
M_{c}(p) := \# \bigg\{ z\in \mathcal{O}_{K} \slash p\mathcal{O}_{K} : \varphi_{(p-1)^{\ell},c}(z) - z \equiv 0 \ \text{(mod $p\mathcal{O}_{K}$)}\bigg\}.
\end{equation}\noindent Setting $\ell =1$ and so $\varphi_{(p-1)^{\ell}, c} = \varphi_{p-1,c}$, we then first prove the following theorem and its generalization \ref{3.2}:

\begin{thm} \label{3.1}
Let $K\slash \mathbb{Q}$ be any number field of degree $n\geq 2$ with the ring of integers $\mathcal{O}_{K}$ and in which $5$ is inert. Let  $\varphi_{4, c}$ be defined by $\varphi_{4, c}(z) = z^4 + c$ for all $c, z\in\mathcal{O}_{K}$, and let $M_{c}(5)$ be as in \textnormal{(\ref{M_{c}})}. Then $M_{c}(5) = 1$ or $2$ for every coefficient $c\equiv 1 \ (mod \ 5\mathcal{O}_{K})$ or $c\in 5\mathcal{O}_{K}$, resp.; otherwise $M_{c}(5) = 0$ for every $c\equiv -1\ (mod \ 5\mathcal{O}_{K})$. 
\end{thm}

\begin{proof}
Let $g(z)= \varphi_{4,c}(z)-z = z^4 - z + c$ and note that for every coefficient $c\in 5\mathcal{O}_{K}$, reducing $g(z)$ modulo prime ideal $5\mathcal{O}_{K}$, we then obtain $g(z)\equiv z^4 - z$ (mod $5\mathcal{O}_{K}$); and so the reduced polynomial $g(z)$ modulo $5\mathcal{O}_{K}$ is now a polynomial defined over a finite field $\mathcal{O}_{K}\slash 5\mathcal{O}_{K}$ of $5^{[K:\mathbb{Q}]} = 5^n$ distinct elements. Now recall from a well-known fact about subfields of finite fields that every subfield of $\mathcal{O}_{K}\slash 5\mathcal{O}_{K}$ is of order $5^r$ for some positive integer $r\mid n$, we then obtain the inclusion $\mathbb{F}_{5}\hookrightarrow\mathcal{O}_{K}\slash 5\mathcal{O}_{K}$ of fields, where $\mathbb{F}_{5}$ is a field of order 5. Moreover, it is a well-known fact about polynomials over finite fields that the quartic monic polynomial $h(x)=x^4-1$ has $4$ distinct nonzero roots in $\mathbb{F}_{5}$; and so $z^4=1$ for every nonzero element $z\in \mathbb{F}_{5}$. But now notice that the reduced polynomial $g(z)\equiv 1 - z$ (mod $5\mathcal{O}_{K}$) for every nonzero $z\in \mathbb{F}_{5}\subset\mathcal{O}_{K}\slash 5\mathcal{O}_{K}$; and from which it then also follows by the factor theorem that $g(z)$ modulo $5\mathcal{O}_{K}$ has a nonzero root in $\mathbb{F}_{5}\subset \mathcal{O}_{K}\slash 5\mathcal{O}_{K}$, namely, $z\equiv 1$ (mod $5\mathcal{O}_{K}$). Moreover, since $z$ is also a linear factor of the reduced polynomial $g(z)\equiv z(z^3 - 1)$ (mod $5\mathcal{O}_{K}$), it then follows $z\equiv 0$ (mod $5\mathcal{O}_{K}$) is also a root of $g(z)$ modulo $5\mathcal{O}_{K}$ in $\mathcal{O}_{K}\slash 5\mathcal{O}_{K}$. But now we then conclude $\#\{ z\in \mathcal{O}_{K}\slash5\mathcal{O}_{K} : \varphi_{4,c}(z) - z \equiv 0 \text{ (mod $5\mathcal{O}_{K}$)}\} = 2$ and so the number $M_{c}(5) = 2$. To also see $M_{c}(5) = 1$ for every coefficient $c\equiv 1$ (mod $5\mathcal{O}_{K}$), we note that with $c\equiv 1$ (mod  $5\mathcal{O}_{K}$) and also since $z^4 = 1$ for every nonzero $z\in \mathbb{F}_{5}\subset \mathcal{O}_{K}\slash5\mathcal{O}_{K}$, then reducing $g(z)= \varphi_{4,c}(z)-z = z^4 - z + c$ modulo $5\mathcal{O}_{K}$, we then obtain $g(z)\equiv 2 - z$ (mod $5\mathcal{O}_{K}$) for every nonzero point $z\in \mathbb{F}_{5}\subset \mathcal{O}_{K}\slash5\mathcal{O}_{K}$; and so $g(z)$ (mod $5\mathcal{O}_{K}$) has a root in $\mathcal{O}_{K}\slash5\mathcal{O}_{K}$, namely, $z\equiv 2$ (mod $5\mathcal{O}_{K}$); and so we then conclude $M_{c}(5) = 1$. Finally, to see $M_{c}(5) = 0$ for every coefficient $c\equiv -1$ (mod $ 5\mathcal{O}_{K})$, we note that since $c \equiv -1$ (mod $5\mathcal{O}_{K}$) and also since $z^4 = 1$ for every nonzero $z\in \mathbb{F}_{5}\subset \mathcal{O}_{K}\slash5\mathcal{O}_{K}$, we then obtain $z^4 - z + c\equiv -z$ (mod $5\mathcal{O}_{K}$) and so $g(z) \equiv -z$ (mod $5\mathcal{O}_{K}$). But now notice $z\equiv 0$ (mod $5\mathcal{O}_{K}$) is a root of $g(z)$ modulo $5\mathcal{O}_{K}$ for every coefficient $c\equiv -1$ (mod $5\mathcal{O}_{K}$) and $c\equiv 0$ (mod $5\mathcal{O}_{K}$) as seen from the first part; which then clearly is impossible, since $-1\not \equiv 0$ (mod $5\mathcal{O}_{K}$). This then means $g(x)=\varphi_{4,c}(x)-x$ has no roots in $\mathcal{O}_{K}\slash 5\mathcal{O}_{K}$ for every coefficient $c\equiv -1$ (mod $5\mathcal{O}_{K}$); and so we then conclude $M_{c}(5) = 0$, as desired.
\end{proof} 
We now wish to generalize Theorem \ref{3.1} to any polynomial map $\varphi_{p-1, c}$ for any prime $p\geq 5$. More precisely, we prove that the number of distinct fixed points of every $\varphi_{p-1, c}$ modulo $p\mathcal{O}_{K}$ is equal to $1$ or $2$ or $0$:

\begin{thm} \label{3.2}
Let $K\slash \mathbb{Q}$ be any number field of degree $n\geq 2$ with the ring $\mathcal{O}_{K}$, and in which any fixed prime $p\geq 5$ is inert. Let $\varphi_{p-1, c}$ be defined by $\varphi_{p-1, c}(z) = z^{p-1} + c$ for all $c, z\in\mathcal{O}_{K}$, and $M_{c}(p)$ be as in \textnormal{(\ref{M_{c}})}. Then $M_{c}(p) = 1$ or $2$ for every $c\equiv 1 \ (mod \ p\mathcal{O}_{K})$ or $c\in p\mathcal{O}_{K}$, resp.; otherwise $M_{c}(p) = 0$ for any $c\equiv -1\ (mod \ p\mathcal{O}_{K})$.
\end{thm}
\begin{proof}
By applying a similar argument as in the Proof of Theorem \ref{3.1}, we then obtain the count as desired. That is, let $g(z)= \varphi_{p-1,c}(z)-z = z^{p-1} - z + c$ and note that for every coefficient $c\in p\mathcal{O}_{K}$, reducing $g(z)$ modulo prime ideal $p\mathcal{O}_{K}$, we then obtain $g(z)\equiv z^{p-1} - z$ (mod $p\mathcal{O}_{K}$); and so the reduced polynomial $g(z)$ modulo $p\mathcal{O}_{K}$ is now a polynomial defined over a finite field $\mathcal{O}_{K}\slash p\mathcal{O}_{K}$ of order $p^n$. But now, as before we have the inclusion $\mathbb{F}_{p}\hookrightarrow\mathcal{O}_{K}\slash p\mathcal{O}_{K}$ of fields, where $\mathbb{F}_{p}$ is a field of order $p$. Moreover, it is a well-known fact the monic polynomial $h(x)=x^{p-1} -1$ vanishes at $p-1$ distinct nonzero points in $\mathbb{F}_{p}$; and so $z^{p-1} = 1$ for every nonzero $z\in \mathbb{F}_{p}$. But now $g(z)\equiv 1 - z$ (mod $p\mathcal{O}_{K}$) for every nonzero $z\in \mathbb{F}_{p}\subset\mathcal{O}_{K}\slash p\mathcal{O}_{K}$; and so $g(z)$ has a nonzero root in $\mathbb{F}_{p}$ and hence in $\mathcal{O}_{K}\slash p\mathcal{O}_{K}$, namely, $z\equiv 1$ (mod $p\mathcal{O}_{K}$). Moreover, since $z$ is a linear factor of $g(z)\equiv z(z^{p-2} - 1)$ (mod $p\mathcal{O}_{K}$), it then  follows $z\equiv 0$ (mod $p\mathcal{O}_{K}$) is also a root of $g(z)$ modulo $p\mathcal{O}_{K}$ in $\mathcal{O}_{K}\slash p\mathcal{O}_{K}$. But now we then conclude $\#\{ z\in \mathcal{O}_{K}\slash p\mathcal{O}_{K} : \varphi_{p-1,c}(z) - z \equiv 0 \text{ (mod $p\mathcal{O}_{K}$)}\} = 2$ and so the number $M_{c}(p) = 2$. To see $M_{c}(p) = 1$ for every coefficient $c\equiv 1$ (mod $p\mathcal{O}_{K}$), we note that with $c\equiv 1$ (mod $p\mathcal{O}_{K}$) and also since $z^{p-1} = 1$ for every nonzero $z\in \mathbb{F}_{p}\subset \mathcal{O}_{K}\slash p\mathcal{O}_{K}$, then reducing $g(z)= \varphi_{p-1,c}(z)-z = z^{p-1} - z + c$ modulo $p\mathcal{O}_{K}$, we then obtain $g(z)\equiv 2 - z$ (mod $p\mathcal{O}_{K}$) and so $g(z)$ (mod $p\mathcal{O}_{K}$) has a root in $\mathcal{O}_{K}\slash p\mathcal{O}_{K}$, namely, $z\equiv 2$ (mod $p\mathcal{O}_{K}$); and so we then conclude $M_{c}(p) = 1$. Finally, to see $M_{c}(p) = 0$ for every coefficient $c\equiv -1$ (mod $ p\mathcal{O}_{K})$, we note that since $c \equiv -1$ (mod $p\mathcal{O}_{K}$) and also since $z^{p-1} = 1$ for every nonzero $z\in \mathbb{F}_{p}\subset \mathcal{O}_{K}\slash p\mathcal{O}_{K}$, it then follows $z^{p-1} - z + c\equiv -z$ (mod $p\mathcal{O}_{K}$); and so $g(z) \equiv -z$ (mod $p\mathcal{O}_{K}$). But now, as before we notice $z\equiv 0$ (mod $p\mathcal{O}_{K}$) is a root of $g(z)$ modulo $p\mathcal{O}_{K}$ in $\mathcal{O}_{K}\slash p\mathcal{O}_{K}$ for every $c\equiv -1$ (mod $p\mathcal{O}_{K}$) and $c\equiv 0$ (mod $p\mathcal{O}_{K}$) as seen from the first part; which then is not possible, since $-1\not \equiv 0$ (mod $p\mathcal{O}_{K}$). This then means $g(x)=\varphi_{p-1,c}(x)-x$ has no roots in $\mathcal{O}_{K}\slash p\mathcal{O}_{K}$ for every coefficient $c\equiv -1$ (mod $p\mathcal{O}_{K}$); and so we then conclude $M_{c}(p) = 0$, as needed.
\end{proof}

Now wish to generalize Theorem \ref{3.2} further to any $\varphi_{(p-1)^{\ell}, c}$ for any prime $p\geq 5$ and any $\ell\in \mathbb{Z}_{\geq 1}$. That is, we prove that the number of distinct integral fixed points of any $\varphi_{(p-1)^{\ell}, c}$ modulo $p\mathcal{O}_{K}$ is also $1$ or $2$ or zero:

\begin{thm} \label{3.3}
Let $K\slash \mathbb{Q}$ be any number field of degree $n\geq 2$ with $\mathcal{O}_{K}$, and in which any fixed prime $p\geq 5$ is inert. Let $\varphi_{(p-1)^{\ell}, c}$ be defined by $\varphi_{(p-1)^{\ell}, c}(z) = z^{(p-1)^{\ell}} + c$ for all $c, z\in\mathcal{O}_{K}$ and $\ell\in \mathbb{Z}_{\geq 1}$. Let $M_{c}(p)$ be as in \textnormal{(\ref{M_{c}})}. Then $M_{c}(p) = 1$ or $2$ for any $c\equiv 1 \ (mod \ p\mathcal{O}_{K})$ or $c\in p\mathcal{O}_{K}$, resp.; else $M_{c}(p) = 0$ for any $c\equiv -1\ (mod \ p\mathcal{O}_{K})$.
\end{thm}
\begin{proof}
Applying a similar argument as in Proof of Theorem \ref{3.2}, we then obtain the count as desired. That is, let $g(z)= \varphi_{(p-1)^{\ell},c}(z) - z=z^{(p-1)^{\ell}} - z + c$ and note that for every coefficient $c \in p\mathcal{O}_{K}$, reducing $g(z)$ modulo prime $p\mathcal{O}_{K}$, it then follows $g(z)\equiv z^{(p-1)^{\ell}} - z$ (mod $p\mathcal{O}_{K}$); and so $g(z)$ modulo $p\mathcal{O}_{K}$ is now a polynomial defined over a field $\mathcal{O}_{K}\slash p\mathcal{O}_{K}$. But now recall the inclusion $\mathbb{F}_{p}\hookrightarrow\mathcal{O}_{K}\slash p\mathcal{O}_{K}$ of fields and also recall $z^{p-1} = 1$ for every nonzero $z\in \mathbb{F}_{p}$, it then also follows $z^{(p-1)^{\ell}} = 1$ for every nonzero element $z\in \mathbb{F}_{p}$ and for every integer $\ell \geq 1$. But now $g(z)\equiv 1 - z$ (mod $p\mathcal{O}_{K}$) for every nonzero $z\in \mathbb{F}_{p}\subset\mathcal{O}_{K}\slash p\mathcal{O}_{K}$; and so $g(z)$ has a nonzero root in $\mathbb{F}_{p}\subset \mathcal{O}_{K}\slash p\mathcal{O}_{K}$, namely, $z\equiv 1$ (mod $p\mathcal{O}_{K}$). Moreover, since $z$ is a linear factor of $g(z)\equiv z(z^{(p-1)^{\ell}-1} - 1)$ (mod $p\mathcal{O}_{K}$), it then also follows $z\equiv 0$ (mod $p\mathcal{O}_{K}$) is a root of $g(z)$ modulo $p\mathcal{O}_{K}$ in $\mathcal{O}_{K}\slash p\mathcal{O}_{K}$. But now we then conclude $\#\{ z\in \mathcal{O}_{K}\slash p\mathcal{O}_{K} : \varphi_{(p-1)^{\ell},c}(z) - z \equiv 0 \text{ (mod $p\mathcal{O}_{K}$)}\} = 2$ and so $M_{c}(p) = 2$. To see $M_{c}(p) = 1$ for every coefficient $c\equiv 1$ (mod $p\mathcal{O}_{K}$), we note that since $c\equiv 1$ (mod $p\mathcal{O}_{K}$) and since also $z^{(p-1)^{\ell}} = 1$ for every nonzero $z\in \mathbb{F}_{p}\subset \mathcal{O}_{K}\slash p\mathcal{O}_{K}$, then reducing $g(z)= \varphi_{(p-1)^{\ell},c}(z)-z = z^{(p-1)^{\ell}} - z + c$ modulo $p\mathcal{O}_{K}$, it then follows $g(z)\equiv 2 - z$ (mod $p\mathcal{O}_{K}$) and so $g(z)$ (mod $p\mathcal{O}_{K}$) has a root in $\mathcal{O}_{K}\slash p\mathcal{O}_{K}$, namely, $z\equiv 2$ (mod $p\mathcal{O}_{K}$); and so we conclude $M_{c}(p) = 1$. Finally, to see $M_{c}(p) = 0$ for every coefficient $c\equiv -1$ (mod $p\mathcal{O}_{K})$ and for every $\ell \in \mathbb{Z}_{\geq 1}$, we note that since $c \equiv -1$ (mod $p\mathcal{O}_{K}$) and since also $z^{(p-1)^{\ell}} = 1$ for every nonzero $z\in \mathbb{F}_{p}\subset \mathcal{O}_{K}\slash p\mathcal{O}_{K}$ and for every $\ell \in \mathbb{Z}_{\geq 1}$, it then follows $z^{(p-1)^{\ell}} - z + c\equiv -z$ (mod $p\mathcal{O}_{K}$); and so $g(z) \equiv -z$ (mod $p\mathcal{O}_{K}$). But now we note $z\equiv 0$ (mod $p\mathcal{O}_{K}$) is a root of $g(z)$ modulo $p\mathcal{O}_{K}$ for every coefficient $c\equiv -1$ (mod $p\mathcal{O}_{K}$) and $c\equiv 0$ (mod $p\mathcal{O}_{K}$) as seen from the first part; from which we then obtain a contradiction that $-1\equiv 0$ (mod $p\mathcal{O}_{K}$). It then follows that $g(x)=\varphi_{(p-1)^{\ell},c}(x)-x$ has no roots in $\mathcal{O}_{K}\slash p\mathcal{O}_{K}$ for every coefficient $c\equiv -1$ (mod $p\mathcal{O}_{K}$) and for every $\ell \in \mathbb{Z}_{\geq 1}$; and so we then conclude that $M_{c}(p) = 0$. This then completes the whole proof, as desired.
\end{proof}

Restricting on the subring $\mathbb{Z}\subset \mathcal{O}_{K}$ of ordinary integers, we then obtain the following consequence of Theorem \ref{3.3} on the number of integral fixed points of any $\varphi_{(p-1)^{\ell},c}$ modulo $p$ for any prime $p\geq 5$ and $\ell \in \mathbb{Z}_{\geq 1}$:

\begin{cor} \label{cor3.4}
Let $p\geq 5$ be any fixed prime integer, and $\ell \geq 1$ be any integer. Let $\varphi_{(p-1)^{\ell}, c}$ be defined by $\varphi_{(p-1)^{\ell}, c}(z) = z^{(p-1)^{\ell}} + c$ for all $c, z\in\mathbb{Z}$, and let $M_{c}(p)$ be as in \textnormal{(\ref{M_{c}})} with $\mathcal{O}_{K} / p\mathcal{O}_{K}$ replaced with $\mathbb{Z}\slash p\mathbb{Z}$. Then $M_{c}(p) = 1$ or $2$ for any coefficient $c\equiv 1 \ (mod \ p)$ or $c = pt$, resp.; otherwise $M_{c}(p) = 0$ for any $c\equiv -1\ (mod \ p)$.
\end{cor}

\begin{proof}
By applying  a similar argument as in the Proof of Theorem \ref{3.3}, we then obtain the count as desired. That is, let $g(z)= z^{(p-1)^{\ell}} - z + c$ and note that for every coefficient $c$ divisible by $p$, reducing $g(z)$ modulo $p$, it then follows $g(z)\equiv z^{(p-1)^{\ell}} - z$ (mod $p$); and so $g(z)$ modulo $p$ is now a polynomial defined over a field $\mathbb{Z}\slash p\mathbb{Z}$ of $p$ distinct elements. So now, recall that the monic polynomial $h(x)=x^{p-1} -1$ has $p-1$ distinct nonzero roots in $\mathbb{Z}\slash p\mathbb{Z}$; and so $z^{p-1} = 1 = z^{(p-1)^{\ell}}$ for every nonzero $z\in \mathbb{Z}\slash p\mathbb{Z}$ and every $\ell \in \mathbb{Z}_{\geq 1}$. But now $g(z)\equiv 1 - z$ (mod $p$) for every nonzero $z\in \mathbb{Z}\slash p\mathbb{Z}$ and so $g(z)$ has a nonzero root in $\mathbb{Z}\slash p\mathbb{Z}$, namely, $z\equiv 1$ (mod $p$). Moreover, since $z$ is a linear factor of $g(z)\equiv z(z^{(p-1)^{\ell}-1} - 1)$ (mod $p$), it then follows $z\equiv 0$ (mod $p$) is a root of $g(z)$ modulo $p$. But now we then conclude $\#\{ z\in \mathbb{Z}\slash p\mathbb{Z} : \varphi_{(p-1)^{\ell},c}(z) - z \equiv 0 \text{ (mod $p$)}\} = 2$ and so $M_{c}(p) = 2$. To see $M_{c}(p) = 1$ for every coefficient $c\equiv 1$ (mod $p$), we note that with $c\equiv 1$ (mod $p$) and since also $z^{(p-1)^{\ell}} = 1$ for every nonzero $z\in \mathbb{Z}\slash p\mathbb{Z}$ and $\ell \in \mathbb{Z}_{\geq 1}$, then reducing $g(z)= \varphi_{(p-1)^{\ell},c}(z)-z = z^{(p-1)^{\ell}} - z + c$ modulo $p$, we then obtain $g(z)\equiv 2 - z$ (mod $p$) and so $g(z)$ modulo $p$ has a root in $\mathbb{Z}\slash p\mathbb{Z}$, namely, $z\equiv 2$ (mod $p$); and so we conclude $M_{c}(p) = 1$. Finally, to see $M_{c}(p) = 0$ for every coefficient $c\equiv -1$ (mod $p)$ and for every $\ell \geq 1$, we note that since $c \equiv -1$ (mod $p$) and since also $z^{(p-1)^{\ell}} = 1$ for every nonzero $z\in \mathbb{Z}\slash p\mathbb{Z}$ and for every $\ell \geq 1$, we then obtain $z^{(p-1)^{\ell}} - z + c\equiv -z$ (mod $p$); and so $g(z) \equiv -z$ (mod $p$). But now we note $z\equiv 0$ (mod $p$) is a root of $g(z)$ modulo $p$ for every coefficient $c\equiv -1$ (mod $p$) and $c\equiv 0$ (mod $p$) as seen from the first part; from which we then obtain a contradiction that $-1 \equiv 0$ (mod $p$). This then means $g(x)=\varphi_{(p-1)^{\ell},c}(x)-x$ has no roots in $\mathbb{Z}\slash p\mathbb{Z}$ for every $c\equiv -1$ (mod $p$) and for every $\ell \in \mathbb{Z}_{\geq 1}$; and so we then conclude $M_{c}(p) = 0$, as also needed.
\end{proof}

\begin{rem}
With now Theorem \ref{3.3}, we may also to each distinct integral fixed point of $\varphi_{(p-1)^{\ell},c}$ associate an integral fixed orbit. In doing so, we then also obtain a dynamical translation of Theorem \ref{3.3}, namely, that the number of distinct integral fixed orbits of any $\varphi_{(p-1)^{\ell},c}$ iterated on the space $\mathcal{O}_{K} / p\mathcal{O}_{K}$ is $1$ or $2$ or $0$. Furthermore, as noted in Intro.\ref{sec1} that in all the coefficient cases $c\equiv \pm 1, 0$ (mod $p\mathcal{O}_{K})$ considered in Theorem \ref{3.3}, the count obtained on the number of distinct integral fixed points of any $\varphi_{(p-1)^{\ell},c}$ modulo $p\mathcal{O}_{K}$ is independent of both $p$ (and so independent of degree of $\varphi_{(p-1)^{\ell},c}$ for any $\ell \in \mathbb{Z}_{\geq 1}$) and $n=[K:\mathbb{Q}]$. Moreover, the expected total count (namely, $1+2+0 =3$) in Theorem \ref{3.3} on the number of distinct integral fixed points in the whole family of maps $\varphi_{(p-1)^{\ell},c}$ modulo $p\mathcal{O}_{K}$ is not only also independent of $p$ (and so independent of deg$(\varphi_{(p-1)^{\ell},c})$) and $n$, but is also a constant $3$ even as $(p-1)^{\ell}\to \infty$ or $n\to \infty$; which somewhat surprising coincides with a phenomenon noted on expected total number three of distinct fixed points in the whole family of maps $\varphi_{p^{\ell},c}$ modulo $p\mathcal{O}_{K}$ in Remark \ref{re2.5}. Notice in Corollary \ref{cor3.4} that the expected total number of distinct integral fixed points in the whole family of maps $\varphi_{(p-1)^{\ell},c}$ modulo $p$ is equal to $1 + 2 + 0 =3$; which somewhat surprising coincides with the second prediction on the upper bound observed in Remark \ref{nt} on Hutz's Conjecture \ref{conjecture 3.2.1}.
\end{rem}

\section{On the Average Number of Fixed Points of any Polynomial Map $\varphi_{p^{\ell},c}$ \& $\varphi_{(p-1)^{\ell},c}$}\label{sec4}

In this section, we wish to inspect independently the behavior of the fixed point-counting functions $N_{c}(p)$ and $M_{c}(p)$ as $c$ tends to infinity. First, we wish to determine: \say{\textit{What is the average value of the function $N_{c}(p)$ as $c \to \infty$?}} The following corollary shows that the average value of $N_{c}(p)$ exists and is equal to $3$ or $0$ as $c\to \infty$:
\begin{cor}\label{cor4.1}
Let $K\slash \mathbb{Q}$ be a real number field of any degree $n\geq 2$ with the ring of integers $\mathcal{O}_{K}$, and in which a prime $p\geq 3$ is inert. Then the average value of $N_{c}(p)$ exists and is $3$ or $0$ as $c\to\infty$. That is, we have
\begin{myitemize}
    \item[\textnormal{(a)}] \textnormal{Avg} $N_{c = pt}(p) := \lim\limits_{c\to\infty} \Large{\frac{\sum\limits_{3\leq p\leq c, \ p\mid c \text{ in } \mathcal{O}_{K}}N_{c}(p)}{\Large{\sum\limits_{3\leq p\leq c, \ p\mid c \text{ in } \mathcal{O}_{K}}1}}} = 3.$  
    
    \item[\textnormal{(b)}] \textnormal{Avg} $N_{c\neq pt}(p):= \lim\limits_{c \to\infty} \Large{\frac{\sum\limits_{3\leq p\leq c, \ p\nmid c \text{ in } \mathcal{O}_{K}}N_{c}(p)}{\Large{\sum\limits_{3\leq p\leq c, \ p\nmid c \text{ in } \mathcal{O}_{K}}1}}} =  0$.    
\end{myitemize}

\end{cor}
\begin{proof}
Since from Theorem \ref{2.3} we know that the number $N_{c}(p) = 3$ for every prime $p\mid c$ in $\mathcal{O}_{K}$, we then obtain  $\lim\limits_{c\to\infty} \Large{\frac{\sum\limits_{3\leq p\leq c, \ p\mid c \text{ in } \mathcal{O}_{K}}N_{c}(p)}{\Large{\sum\limits_{3\leq p\leq c, \ p\mid c \text{ in } \mathcal{O}_{K}}1}}} = 3\lim\limits_{c\to\infty} \Large{\frac{\sum\limits_{3\leq p\leq c, \ p\mid c \text{ in } \mathcal{O}_{K}}1}{\Large{\sum\limits_{3\leq p\leq c, \ p\mid c \text{ in } \mathcal{O}_{K}}1}}} = 3$; and so the average value of $N_{c}(p)$, namely, Avg $N_{c = pt}(p)$ is equal to $3$, as needed. Similarly, we also know from Theorem \ref{2.3} that $N_{c}(p) = 0$ for any prime $p\nmid c$ in $\mathcal{O}_{K}$, we then obtain $\lim\limits_{c\to\infty} \Large{\frac{\sum\limits_{3\leq p\leq c, \ p\nmid c \text{ in } \mathcal{O}_{K}}N_{c}(p)}{\Large{\sum\limits_{3\leq p\leq c, \ p\nmid c \text{ in } \mathcal{O}_{K}}1}}} = 0$; and so the average value Avg $N_{c\neq pt}(p) = 0$, as also needed.
\end{proof}
\begin{rem} \label{re4.2}
From arithmetic statistics to arithmetic dynamics, we note that Corollary \ref{cor4.1} shows that every map $\varphi_{p^{\ell},c}$ iterated on the space $\mathcal{O}_{K} / p\mathcal{O}_{K}$ has on average three or zero distinct integral fixed orbits as $c\to \infty$.
\end{rem}

Similarly, we also wish to determine: \say{\textit{What is the average value of the function $M_{c}(p)$ as $c \to \infty$?}}. The following corollary shows that the average value of the function $M_{c}(p)$ also exists and is $1$ or $2$ or $0$ as $c\to \infty$:
\begin{cor}\label{cor5.1}
Let $K\slash \mathbb{Q}$ be any number field of degree $n\geq 2$ with the ring of integers $\mathcal{O}_{K}$, and in which a prime $p\geq 5$ is inert. Then the average value of $M_{c}(p)$ is equal to $1$ or $2$ or $0$ as $c\to\infty$. That is, we have 
\begin{myitemize}
    \item[\textnormal{(a)}] \textnormal{Avg} $M_{c-1 = pt}(p) := \lim\limits_{c\to\infty} \Large{\frac{\sum\limits_{5\leq p\leq (c-1), \ p\mid (c-1) \text{ in } \mathcal{O}_{K}}M_{c}(p)}{\Large{\sum\limits_{5\leq p\leq (c-1), \ p\mid (c-1) \text{ in } \mathcal{O}_{K}}1}}} = 1.$ 

    \item[\textnormal{(b)}] \textnormal{Avg} $M_{c= pt}(p) := \lim\limits_{c\to\infty} \Large{\frac{\sum\limits_{5\leq p\leq c, \ p\mid c \text{ in } \mathcal{O}_{K}}M_{c}(p)}{\Large{\sum\limits_{5\leq p\leq c, \ p\mid c \text{ in } \mathcal{O}_{K}}1}}} = 2.$
    
    \item[\textnormal{(c)}] \textnormal{Avg} $M_{c+1= pt}(p):= \lim\limits_{c \to\infty} \Large{\frac{\sum\limits_{5\leq p\leq (c+1), \ p\mid (c+1) \text{ in } \mathcal{O}_{K}}M_{c}(p)}{\Large{\sum\limits_{5\leq p\leq (c+1), \ p\mid (c+1) \text{ in } \mathcal{O}_{K}}1}}} =  0$.    
\end{myitemize}

\end{cor}
\begin{proof}
Since from Theorem \ref{3.3} we know that $M_{c}(p) = 1$ for any prime $p$ such that $p\mid (c-1)$ in $\mathcal{O}_{K}$, we then obtain $\lim\limits_{c\to\infty} \Large{\frac{\sum\limits_{5\leq p\leq (c-1), \ p\mid (c-1) \text{ in } \mathcal{O}_{K}}M_{c}(p)}{\Large{\sum\limits_{5\leq p\leq (c-1), \ p\mid (c-1) \text{ in } \mathcal{O}_{K}}1}}} = \lim\limits_{c\to\infty} \Large{\frac{\sum\limits_{5\leq p\leq (c-1), \ p\mid (c-1)\text{ in } \mathcal{O}_{K}}1}{\Large{\sum\limits_{5\leq p\leq (c-1), \ p\mid (c-1) \text{ in } \mathcal{O}_{K}}1}}} = 1$; and so the average value of $M_{c}(p)$, namely, Avg $M_{c-1 = pt}(p) = 1$. Similarly, since from Theorem \ref{3.3} we know that $M_{c}(p) = 2$ or $0$ for any prime $p$ such that $p\mid c$ or $p\mid (c+1)$ in $\mathcal{O}_{K}$, resp., we then obtain $\lim\limits_{c\to\infty} \Large{\frac{\sum\limits_{5\leq p\leq c, \ p\mid c \text{ in } \mathcal{O}_{K}}M_{c}(p)}{\Large{\sum\limits_{5\leq p\leq c, \ p\mid c \text{ in } \mathcal{O}_{K}}1}}} = 2$ or $\lim\limits_{c\to\infty} \Large{\frac{\sum\limits_{5\leq p\leq (c+1), \ p\mid (c+1) \text{ in } \mathcal{O}_{K}}M_{c}(p)}{\Large{\sum\limits_{5\leq p\leq (c+1), \ p\mid (c+1) \text{ in } \mathcal{O}_{K}}1}}} = 0$; and thus Avg $M_{c = pt}(p) = 2$ or Avg $M_{c+1 = pt}(p) = 0$, resp., as desired.
\end{proof} 
\begin{rem} \label{re4.4}
As before, from arithmetic statistics to arithmetic dynamics, we also note that Corollary \ref{cor5.1} shows that any $\varphi_{(p-1)^{\ell},c}$ iterated on the space $\mathcal{O}_{K} / p\mathcal{O}_{K}$ has on average one or two or no fixed orbits as $c\to \infty$.
\end{rem}

\section{On the Density of Integer Polynomials $\varphi_{p^{\ell},c}(x)\in \mathcal{O}_{K}[x]$ with the Number $N_{c}(p) = 3$}\label{sec5}
In this and the next sections, we first focus our counting on $\mathbb{Z}\subset \mathcal{O}_{K}$ and then determine: \say{\textit{For any fixed $\ell \in \mathbb{Z}_{\geq 1}$, what is the density of integer polynomials $\varphi_{p^{\ell},c}(x) \in \mathcal{O}_{K}[x]$ with three fixed points modulo $p\mathbb{Z}$?}} The following corollary shows that there are very few polynomials $\varphi_{p^{\ell},c}(x)\in \mathbb{Z}[x]$ having $3$ fixed points modulo $p\mathbb{Z}$:
\begin{cor}\label{5.1}
Let $K\slash \mathbb{Q}$ be a real number field of any degree $n\geq 2$ with the ring of integers $\mathcal{O}_{K}$, and in which any prime $p\geq 3$ is inert. Let $\ell \geq 1$ be any fixed integer. Then the density of monic integer polynomials $\varphi_{p^{\ell},c}(x) = x^{p^{\ell}} + c\in \mathcal{O}_{K}[x]$ with $N_{c}(p) = 3$ exists and is equal to $0 \%$ as $c\to \infty$. More precisely, we have 
\begin{center}
    $\lim\limits_{c\to\infty} \Large{\frac{\# \{\varphi_{p^{\ell},c}(x)\in \mathbb{Z}[x] \ : \ 3\leq p\leq c \ \text{and} \ N_{c}(p) \ = \ 3\}}{\Large{\# \{\varphi_{p^{\ell},c}(x) \in \mathbb{Z}[x] \ : \ 3\leq p\leq c \}}}} = \ 0.$
\end{center}
\end{cor}
\begin{proof}
Since the defining condition $N_{c}(p) = 3$ is as we proved in Theorem \ref{2.3} and so in Corollary \ref{cor2.4} determined whenever the coefficient $c$ is divisible by $p$, we may then count $\# \{\varphi_{p^{\ell},c}(x) \in \mathbb{Z}[x] : 3\leq p\leq c \ \text{and} \ N_{c}(p) \ = \ 3\}$ by simply counting $\# \{\varphi_{p^{\ell},c}(x)\in \mathbb{Z}[x] : 3\leq p\leq c \ \text{and} \ p\mid c \ \text{for \ any \ fixed} \ c \}$. But now applying a similar argument as in [\cite{BK1}, Proof of Corollary 5.1], we then obtain that the limit exits and is equal to $0$, as desired.
\end{proof}\noindent Note that we may also interpret Corollary \ref{5.1} as saying that for any fixed $\ell \in \mathbb{Z}_{\geq 1}$, the probability of choosing randomly a monic $\varphi_{p^{\ell},c}(x)$ in the space $\mathbb{Z}[x]\subset \mathcal{O}_{K}[x]$ having $3$ distinct integral fixed points modulo $p\mathbb{Z}$ is zero.

\section{On Densities of Monic Integer Polynomials $\varphi_{(p-1)^{\ell},c}(x)\in \mathcal{O}_{K}[x]$ with $M_{c}(p) = 1$ or $2$}\label{sec6}

As in Section \ref{sec5}, we also wish to determine: \say{\textit{For any fixed $\ell \in \mathbb{Z}_{\geq 1}$, what is the density of integer polynomials $\varphi_{(p-1)^{\ell},c}(x) = x^{(p-1)^{\ell}} + c\in \mathcal{O}_{K}[x]$ with two integral fixed points modulo $p\mathbb{Z}$?}} The following corollary shows that there are also very few monic polynomials $\varphi_{(p-1)^{\ell},c}(x)\in \mathbb{Z}[x]$ having two integral fixed points modulo $p\mathbb{Z}$:
\begin{cor}\label{6.1}
Let $K\slash \mathbb{Q}$ be any number field of degree $n\geq 2$ with the ring of integers $\mathcal{O}_{K}$, and in which any prime $p\geq 5$ is inert. Let $\ell \geq 1$ be any fixed integer. Then the density of monic integer polynomials $\varphi_{(p-1)^{\ell},c}(x) = x^{(p-1)^{\ell}} + c\in \mathcal{O}_{K}[x]$ with $M_{c}(p) = 2$ exists and is equal to $0 \%$ as $c\to \infty$. Specifically, we have 
\begin{center}
    $\lim\limits_{c\to\infty} \Large{\frac{\# \{\varphi_{(p-1)^{\ell},c}(x) \in \mathbb{Z}[x]\ : \ 5\leq p\leq c \ and \ M_{c}(p) \ = \ 2\}}{\Large{\# \{\varphi_{(p-1)^{\ell},c}(x) \in \mathbb{Z}[x]\ : \ 5\leq p\leq c \}}}} = \ 0.$
\end{center}
\end{cor}
\begin{proof}
By applying a similar argument as in the Proof of Corollary \ref{5.1}, we then obtain the limit as desired.
\end{proof} \noindent As before, we may also interpret Cor. \ref{6.1} by saying that for any fixed $\ell \in \mathbb{Z}_{\geq 1}$, the probability of choosing randomly a monic polynomial $\varphi_{(p-1)^{\ell},c}(x)$ in $\mathbb{Z}[x]\subset \mathcal{O}_{K}[x]$ with two integral fixed points modulo $p\mathbb{Z}$ is zero.

The following immediate corollary shows that for any fixed $\ell \in \mathbb{Z}_{\geq 1}$, the probability of choosing randomly a polynomial $\varphi_{(p-1)^{\ell},c}(x)=x^{(p-1)^{\ell}}+c$ in the space $\mathbb{Z}[x]\subset \mathcal{O}_{K}[x]$ with one fixed point modulo $p\mathbb{Z}$ is also zero:

\begin{cor}\label{6.2}
Let $K\slash \mathbb{Q}$ be any number field of degree $n\geq 2$ with the ring of integers $\mathcal{O}_{K}$, and in which any prime $p\geq 5$ is inert. Let $\ell \geq 1$ be any fixed integer. The density of monic integer polynomials $\varphi_{(p-1)^{\ell},c}(x) = x^{(p-1)^{\ell}} + c\in \mathcal{O}_{K}[x]$ with $M_{c}(p) = 1$ exists and is equal to $0 \%$ as $c\to \infty$. More precisely, we have  
\begin{center}
    $\lim\limits_{c\to\infty} \Large{\frac{\# \{\varphi_{(p-1)^{\ell},c}(x) \in \mathbb{Z}[x]\ : \ 5\leq p\leq c \ and \ M_{c}(p) \ = \ 1\}}{\Large{\# \{\varphi_{(p-1)^{\ell},c}(x) \in \mathbb{Z}[x]\ : \ 5\leq p\leq c \}}}} = \ 0.$
\end{center}
\end{cor}
\begin{proof}
As before, $M_{c}(p) = 1$ is as we proved in Corollary \ref{cor3.4} determined whenever the coefficient $c$ is such that $c-1$ is divisible by a prime $p\geq 5$; and so we may count $\# \{\varphi_{(p-1)^{\ell},c}(x) \in \mathbb{Z}[x] : 5\leq p\leq c \ \text{and} \ M_{c}(p) \ = \ 1\}$ by simply counting $\# \{\varphi_{(p-1)^{\ell},c}(x)\in \mathbb{Z}[x] : 5\leq p\leq c \ \text{and} \ p\mid (c-1) \ \text{for \ any \ fixed} \ c \}$. Now applying a similar argument as in [\cite{BK1}, Proof of Corollary 6.2], we then obtain that the limit exists and is equal to $0$, as desired.  
\end{proof}

\section{The Density of Integer Monics $\varphi_{p^{\ell},c}(x)$ with $N_{c}(p) = 0$ \& $\varphi_{(p-1)^{\ell},c}(x)$ with $M_{c}(p) = 0$}\label{sec7}
Recall in Cor. \ref{5.1} that a density of $0\%$ of monic integer (and hence rational) polynomials $\varphi_{p,c}(x)$ have three integral fixed points modulo $p$; and so the density of polynomials $\varphi_{p^{\ell},c}(x)-x\in \mathbb{Z}[x]$ that are reducible modulo $p$ is $0\%$. Now as in Sect.\ref{sec5}, we also wish to determine: \say{\textit{For any fixed $\ell \in \mathbb{Z}_{\geq 1}$, what is the density of  polynomials $\varphi_{p^{\ell},c}(x)\in \mathbb{Z}[x]$ with no fixed points modulo $p$?}} The following corollary shows that the probability of choosing randomly a monic polynomial $\varphi_{p^{\ell},c}(x)\in \mathbb{Z}[x]$ such that $\mathbb{Q}[x]\slash (\varphi_{p^{\ell}, c}(x)-x)$ is a number field of degree $p^{\ell}$ is 1: 
\begin{cor}\label{7.1}
Let $K\slash \mathbb{Q}$ be a real number field of any degree $n\geq 2$ with the ring of integers $\mathcal{O}_{K}$, and in which any prime $p\geq 3$ is inert. Let $\ell \geq 1$ be any fixed integer. Then the density of monic integer polynomials $\varphi_{p^{\ell},c}(x)=x^{p^{\ell}} + c\in \mathcal{O}_{K}[x]$ with $N_{c}(p) = 0$ exists and is equal to $100 \%$ as $c\to \infty$. More precisely, we have 
\begin{center}
    $\lim\limits_{c\to\infty} \Large{\frac{\# \{\varphi_{p^{\ell},c}(x)\in \mathbb{Z}[x] \ : \ 3\leq p\leq c \ and \ N_{c}(p) \ = \ 0 \}}{\Large{\# \{\varphi_{p^{\ell},c}(x) \in \mathbb{Z}[x] \ : \ 3\leq p\leq c \}}}} = \ 1.$
\end{center}
\end{cor}
\begin{proof}
Since the number $N_{c}(p) = 3$ or $0$ for any given prime integer $p\geq 3$ and since we also proved the density in Corollary \ref{5.1}, we then obtain the desired density (i.e., we obtain that the limit exists and is equal to 1). 
\end{proof}
\noindent The foregoing corollary may also be viewed as saying that for any fixed $\ell \in \mathbb{Z}_{\geq 1}$, there are infinitely many monic polynomials $\varphi_{p^{\ell},c}(x) \in \mathbb{Z}[x]\subset \mathbb{Q}[x]$ such that for any monic polynomial $f(x) = \varphi_{p^{\ell},c}(x)-x = x^{p^{\ell}}-x+c$, the quotient ring $\mathbb{Q}_{f} = \mathbb{Q}[x]\slash (f(x))$ induced by $f$ is an algebraic number field of odd prime power degree $p^{\ell}$. Comparing the densities in Corollaries \ref{5.1} and \ref{7.1}, we may also observe that in the whole family of monic integer polynomials $\varphi_{p^{\ell},c}(x) = x^{p^{\ell}} +c$, almost all such polynomials $\varphi_{p^{\ell},c}(x)$ have no integral fixed points modulo $p$ (i.e., have no rational roots); and so almost all such monic polynomials $f(x)$ are irreducible over $\mathbb{Q}$. Consequently, this may also imply that the average value of $N_{c}(p)$ in the whole family of polynomials $\varphi_{p^{\ell},c}(x)\in \mathbb{Z}[x]$ is zero. 

Recall in Corollary \ref{6.1} or \ref{6.2} that a density of $0\%$ of monic integer polynomials $\varphi_{(p-1)^{\ell},c}(x)$ have $M_{c}(p) = 2$ or $1$, resp.; and so the density of $\varphi_{(p-1)^{\ell},c}(x)-x\in \mathbb{Z}[x]$ that are reducible modulo $p$ is $0\%$. So now, we also wish to determine:\say{\textit{For any fixed $\ell \in \mathbb{Z}_{\geq 1}$, what is the density of monic integer polynomials $\varphi_{(p-1)^{\ell},c}(x)\in \mathcal{O}_{K}[x]$ with no integral fixed points modulo $p$?}} The corollary below shows that the probability of choosing randomly a polynomial $\varphi_{(p-1)^{\ell},c}(x)\in \mathbb{Z}[x]$ such that $\mathbb{Q}[x]\slash (\varphi_{(p-1)^{\ell}, c}(x)-x)$ is a number field of degree $(p-1)^{\ell}$ is also $1$:
\begin{cor} \label{7.2}
Let $K\slash \mathbb{Q}$ be any number field of degree $n\geq 2$ with the ring of integers $\mathcal{O}_{K}$, and in which any prime $p\geq 5$ is inert. Let $\ell \geq 1$ be any fixed integer. Then the density of monic integer polynomials $\varphi_{(p-1)^{\ell}, c}(x) = x^{(p-1)^{\ell}}+c\in \mathcal{O}_{K}[x]$ with $M_{c}(p) = 0$ exists and is equal to $100 \%$ as $c\to \infty$. That is, we have 
\begin{center}
    $\lim\limits_{c\to\infty} \Large{\frac{\# \{\varphi_{(p-1)^{\ell}, c}(x)\in \mathbb{Z}[x] \ : \ 5\leq p\leq c \ and \ M_{c}(p) \ = \ 0 \}}{\Large{\# \{\varphi_{(p-1)^{\ell},c}(x) \in \mathbb{Z}[x] \ : \ 5\leq p\leq c \}}}} = \ 1.$
\end{center}
\end{cor}
\begin{proof}
Recall that $M_{c}(p) = 1$ or $2$ or $0$ for any given prime $p\geq 5$ and $\ell \in \mathbb{Z}_{\geq 1}$ and since we also proved the densities in Cor. \ref{6.1} and \ref{6.2}, we then obtain the desired density (i.e., we obtain that the desired limit is 1).
\end{proof}
\noindent As before, Corollary \ref{7.2} also shows that for any fixed $\ell \in \mathbb{Z}_{\geq 1}$, there are infinitely many monic polynomials $\varphi_{(p-1)^{\ell},c}(x)$ over $\mathbb{Z}\subset \mathbb{Q}$ such that for any monic polynomial $g(x) = \varphi_{(p-1)^{\ell},c}(x)-x = x^{(p-1)^{\ell}}-x+c$, the quotient ring $\mathbb{Q}_{g} = \mathbb{Q}[x]\slash (g(x))$ induced by $g$ is an algebraic number field of even degree $(p-1)^{\ell}$. As before, if we compare the densities in Corollaries \ref{6.1}, \ref{6.2} and \ref{7.2}, we may again observe that in the whole family of monic integer polynomials $\varphi_{(p-1)^{\ell},c}(x) = x^{(p-1)^{\ell}} +c$, almost all such monics have no integral fixed points modulo $p$ (i.e., have no rational roots); and so almost all monic polynomials $g(x)$ are irreducible over $\mathbb{Q}$. As a result, it may follow that the average value of $M_{c}(p)$ in the whole family of polynomials $\varphi_{(p-1)^{\ell},c}(x)\in \mathbb{Z}[x]$ is also zero.

As always a central theme in algebraic number theory that whenever one is studying an algebraic number field $K$ of some interest, one must simultaneously try to describe very precisely what the associated ring $\mathcal{O}_{K}$ of integers is; and this is because $\mathcal{O}_{K}$ is classically known to describe  naturally the arithmetic of the underlying number field. However, accessing $\mathcal{O}_{K}$ in practice from a computational point of view is known to be an extremely involved problem. So now, in our case here, it then follows $\mathbb{Q}_{f}$ has a ring of integers $\mathcal{O}_{\mathbb{Q}_{f}}$, and moreover by Bhargava-Shankar-Wang \cite{sch1}, we then also obtain the following corollary showing the probability of choosing randomly a monic polynomial $f\in \mathbb{Z}[x]$ arising from a polynomial discrete dynamical system in Section \ref{sec2} (and ascertained by Corollary \ref{7.1}), such that the quotient $\mathbb{Z}[x]\slash (f(x))$ is the ring of integers of $\mathbb{Q}_{f}$, is $\approx 60.7927\%$:

\begin{cor}\label{7.3}
Assume Corollary \ref{7.1}. When monic polynomials $f(x)\in \mathbb{Z}[x]$ are ordered by height $H(f)$ as defined in \textnormal{\cite{sch1}}, the density of polynomials $f(x)$ such that $\mathbb{Z}_{f}=\mathbb{Z}[x]\slash (f(x))$ is the ring of integers of $\mathbb{Q}_{f}$ is $\zeta(2)^{-1}$. 
\end{cor}

\begin{proof}
Since from Corollary \ref{7.1} we know that for any fixed integer $\ell \geq 1$, there are infinitely many polynomials $f(x)\in \mathbb{Z}[x]\subset \mathbb{Q}[x]$ such that $\mathbb{Q}_{f} = \mathbb{Q}[x]\slash (f(x))$ is a number field of odd degree $p^{\ell}$; and moreover associated to $\mathbb{Q}_{f}$, is the ring $\mathcal{O}_{\mathbb{Q}_{f}}$ of integers. This then means that the set of irreducible monic polynomials $f(x)\in \mathbb{Z}[x]$ such that $\mathbb{Q}_{f}$ is a number field of degree $p^{\ell}$ is not empty. But now applying [\cite{sch1}, Theorem 1.2] to the underlying family of polynomials $f\in \mathbb{Z}[x]$ ordered by height $H(f) = |c|^{1\slash p^{\ell}}$ such that the ring of integers $\mathcal{O}_{\mathbb{Q}_{f}} = \mathbb{Z}[x]\slash (f(x))$, it then follows that the density of such monic integer polynomials $f$ is equal to $\zeta(2)^{-1} \approx 60.7927\%$, as needed.  
\end{proof}

As with $\mathbb{Q}_{f}$, every number field $\mathbb{Q}_{g}$ induced by a polynomial $g$, is naturally equipped with the ring of integers $\mathcal{O}_{\mathbb{Q}_{g}}$, and which as before may be difficult to compute in practice. So now, as before we take great advantage of density result in [\cite{sch1}, Theorem 1.2] and then obtain the following corollary showing the probability of choosing randomly a monic polynomial $g\in \mathbb{Z}[x]$ arising from a polynomial discrete dynamical system in Section \ref{sec3} (and ascertained by Cor. \ref{7.2}), such that $\mathbb{Z}[x]\slash (g(x))$ is the ring of integers of $\mathbb{Q}_{g}$, is also $\approx 60.7927\%$:

\begin{cor}
Assume Corollary \ref{7.2}. When monic polynomials $g(x)\in \mathbb{Z}[x]$ are ordered by height $H(g)$ as defined in \textnormal{\cite{sch1}}, the density of polynomials $g(x)$ such that $\mathbb{Z}_{g}=\mathbb{Z}[x]\slash (g(x))$ is the ring of integers of $\mathbb{Q}_{g}$ is $\zeta(2)^{-1}$. 
\end{cor}

\begin{proof}
By applying a similar argument as in the Proof of Corollary \ref{7.3}, we then obtain the density, as needed.
\end{proof}

\section{On the Number of Algebraic Number fields with Bounded Absolute Discriminant}\label{sec8}

\subsection{On Fields $\mathbb{Q}_{f}$ and $\mathbb{Q}_{g}$ with Bounded Absolute Discriminant and Prescribed Galois group}

Recall from Corollary \ref{7.1} that there is an infinite family of irreducible monic integer polynomials $f(x) = x^{p^{\ell}}-x + c$ such that the quotient $\mathbb{Q}_{f}=\mathbb{Q}[x]\slash (f(x))$ associated to $f$ is a number field of odd degree $p^{\ell}$. Moreover, we may also recall from Corollary \ref{7.2} that we can always find an infinite family of irreducible monic integer polynomials $g(x) = x^{(p-1)^{\ell}}-x + c$ such that the field extension $\mathbb{Q}_{g}=\mathbb{Q}[x]\slash (g(x))$ over $\mathbb{Q}$ arising from $g$ is a number field of even degree $(p-1)^{\ell}\geq 4$. So now, we in this section wish to study the problem of counting number fields; a problem that's originally from and is of very serious interest in arithmetic statistics. Inspired (as in \cite{BK1}) by Bhargava-Shankar-Wang (BSW) \cite{sch}, we then wish to count here the number of primitive number fields $\mathbb{Q}_{f}$ induced by irreducible polynomials $f\in \mathbb{Z}[x]$ arising from a polynomial discrete dynamical system in Section \ref{sec2} (and ascertained by Corollary \ref{7.1}), with bounded absolute discriminant. To this end, we then obtain:
\begin{cor}\label{cor6}
Assume Corollary \ref{7.1} and let $\mathbb{Q}_{f} = \mathbb{Q}[x]\slash (f(x))$ be a primitive number field with discriminant $\Delta(\mathbb{Q}_{f})$. Then up to isomorphism classes of number fields, we have that $\# \{ \mathbb{Q}_{f} \ : \ |\Delta(\mathbb{Q}_{f})| < X \} \ll  X^{p^{\ell}\slash(2p^{\ell}-2)}$. 
\end{cor}
\begin{proof}
From Corollary \ref{7.1}, we know that there are infinitely many monics $f(x) = x^{p^{\ell}} - x + c$ over $\mathbb{Z}$ (and so over $\mathbb{Q}$) such that $\mathbb{Q}_{f}$ is a number field of deg$(f) = p^{\ell}$. But now when the underlying number fields $\mathbb{Q}_{f}$ are primitive, we may then apply an argument of Bhargava-Shankar-Wang in [\cite{sch}, Page 2] to show that up to isomorphism classes of number fields the total number of such primitive number fields with $|\Delta(\mathbb{Q}_{f})| < X$, is bounded above by the number of monic integer polynomials $f(x)$ of degree $p^{\ell}$ with height $H(f) \ll X^{1\slash(2p^{\ell}-2)}$ and vanishing subleading coefficient. Since $H(f) = |c|^{1\slash p^{\ell}}$ as by the definition of height in \cite{sch} and so $|c| \ll X^{p^{\ell}\slash(2p^{\ell}-2)}$, we then obtain $\# \{f(x)\in \mathbb{Z}[x] : H(f) \ll X^{1\slash(2p^{\ell}-2)} \} = \# \{f(x)\in \mathbb{Z}[x] : |c| \ll X^{p^{\ell}\slash(2p^{\ell}-2)} \} \ll X^{p^{\ell}\slash(2p^{\ell}-2)}$. Hence, we then obtain $\# \{ \mathbb{Q}_{f} : |\Delta(\mathbb{Q}_{f})| < X \} \ll  X^{p^{\ell}\slash(2p^{\ell}-2)}$ up to isomorphism classes of number fields, as desired.
\end{proof}

By applying a similar argument as in Corollary \ref{cor6}, we then also obtain the following corollary on the number of primitive number fields $\mathbb{Q}_{g}$ induced by irreducible polynomials $g(x)\in \mathbb{Z}[x]$ arising from a polynomial discrete dynamical system in Section \ref{sec3} (and ascertained by Corollary \ref{7.2}), with bounded absolute discriminant:

\begin{cor}\label{8.2}
Assume Corollary \ref{7.2} and let $\mathbb{Q}_{g} = \mathbb{Q}[x]\slash (g(x))$ be a primitive number field with discriminant $\Delta(\mathbb{Q}_{f})$. Then up to isomorphism classes of number fields, we have $\# \{ \mathbb{Q}_{g} : |\Delta(\mathbb{Q}_{g})| < X \} \ll  X^{(p-1)^{\ell}\slash(2(p-1)^{\ell}-2)}$. 
\end{cor}

\begin{proof}
From Corollary \ref{7.2}, we know that there are infinitely many monics $g(x) = x^{(p-1)^{\ell}}-x + c$ over $\mathbb{Z}$ (and so over $\mathbb{Q}$) such that $\mathbb{Q}_{g}$ is a number field of even degree $(p-1)^{\ell}$. So now when the underlying number fields $\mathbb{Q}_{g}$ are primitive, we may again use an argument of Bhargava-Shankar-Wang in [\cite{sch}, Page 2] to show that up to isomorphism classes of number fields the total number of such primitive number fields with $|\Delta(\mathbb{Q}_{g})| < X$, is bounded above by the number of monic integer polynomials $g(x)$ of degree $(p-1)^{\ell}$ with height $H(g) \ll X^{1\slash(2(p-1)^{\ell}-2)}$ and vanishing subleading coefficient. Now since $H(g) = |c|^{1\slash (p-1)^{\ell}}$ and so $|c| \ll X^{(p-1)^{\ell}\slash(2(p-1)^{\ell}-2)}$, we then obtain $\# \{g(x)\in \mathbb{Z}[x] : H(g) \ll X^{1\slash(2(p-1)^{\ell}-2)} \} = \# \{g(x)\in \mathbb{Z}[x] : |c| \ll X^{(p-1)^{\ell}\slash(2(p-1)^{\ell}-2)} \} \ll X^{(p-1)^{\ell}\slash(2(p-1)^{\ell}-2)}$. Hence, the number $\# \{ \mathbb{Q}_{g} : |\Delta(\mathbb{Q}_{g})| < X \} \ll  X^{(p-1)^{\ell}\slash(2(p-1)^{\ell}-2)}$ up to isomorphism classes of number fields.
\end{proof}

We recall in algebraic number theory that a number field $K$ is called \say{\textit{monogenic}} if there exists an algebraic number $\alpha \in K$ such that the ring $\mathcal{O}_{K}$ of integers is the subring $\mathbb{Z}[\alpha]$ generated by $\alpha$ over $\mathbb{Z}$, i.e., $\mathcal{O}_{K}= \mathbb{Z}[\alpha]$. Now recall in Corollary \ref{cor6} that we counted primitive number fields $\mathbb{Q}_{f}$ with absolute discriminant $|\Delta(\mathbb{Q}_{f})| < X$. So now, we wish to count the number of number fields $\mathbb{Q}_{f}$ induced by irreducible monic polynomials $f\in \mathbb{Z}[x]$ arising from a polynomial discrete dynamical system in Section \ref{sec2} (and ascertained by Corollary \ref{7.1}), that are monogenic with $|\Delta(\mathbb{Q}_{f})| < X$ and with Galois group Gal$(\mathbb{Q}_{f}\slash \mathbb{Q})$ equal to the symmetric group $S_{p^{\ell}}$. 
By taking great advantage of a result of Bhargava-Shankar-Wang [\cite{sch1}, Corollary 1.3], we then obtain:

\begin{cor}\label{8.3}
Assume Corollary \ref{7.1}. Then the number of isomorphism classes of algebraic number fields $\mathbb{Q}_{f}$ of degree $p^{\ell}\geq 3$ and with $|\Delta(\mathbb{Q}_{f})| < X$ that are monogenic and have associated Galois group $S_{p^{\ell}}$ is $\gg X^{\frac{1}{2} + \frac{1}{p^{\ell}}}$.
\end{cor}

\begin{proof}
To see this, we recall from Cor. \ref{7.1} the existence of infinitely many polynomials $f(x)$ over $\mathbb{Z}$ (and so over $\mathbb{Q}$) such that $\mathbb{Q}_{f}$ is a degree-$p^{\ell}$ number field. This then means that the set of fields $\mathbb{Q}_{f}$ is not empty. Applying [\cite{sch1}, Cor. 1.3] to the underlying fields $\mathbb{Q}_{f}$ with $|\Delta(\mathbb{Q}_{f})| < X$ that are monogenic and have associated Galois group $S_{p^{\ell}}$, it then follows that the number of isomorphism classes of such fields $\mathbb{Q}_{f}$ is $\gg X^{\frac{1}{2} + \frac{1}{p^{\ell}}}$, as needed.
\end{proof}

Similarly, by again we again taking great advantage of that same result of (BSW)[\cite{sch1}, Corollary 1.3], we then also obtain the following corollary on the number of number fields $\mathbb{Q}_{g}$ induced by irreducible polynomials $g\in \mathbb{Z}[x]$ arising from a polynomial discrete dynamical system in Section \ref{sec3} (and ascertained by Corollary \ref{7.2}), that are monogenic with $|\Delta(\mathbb{Q}_{g})| < X$ and having Galois group Gal$(\mathbb{Q}_{g}\slash \mathbb{Q})$ equal to symmetric group $S_{(p-1)^{\ell}}$: 

\begin{cor}
Assume Corollary \ref{7.2}. The number of isomorphism classes of algebraic number fields $\mathbb{Q}_{g}$ of degree $(p-1)^{\ell}$ and $|\Delta(\mathbb{Q}_{g})| < X$ that are monogenic and have associated Galois group $S_{(p-1)^{\ell}}$ is $\gg X^{\frac{1}{2} + \frac{1}{(p-1)^{\ell}}}$.
\end{cor}

\begin{proof}
By applying a similar argument as in the Proof of Corollary \ref{8.3}, we then obtain the count, as needed.
\end{proof}

\subsection{On Fields $K_{f}$ and $L_{g}$ with Bounded Absolute Discriminant and Prescribed Galois group}

Recall that we proved in Corollary \ref{7.1} the existence of an infinite family of irreducible monic integer polynomials $f(x) = x^{p^{\ell}} - x + c\in \mathbb{Q}[x]\subset K[x]$ for any fixed $\ell \in \mathbb{Z}_{\geq 1}$; and more to this, we may also recall that the second part of Theorem \ref{2.3} (i.e., the part in which we proved $N_{c}(p) = 0$ for every $c\not \equiv 0\ (\text{mod} \ p\mathcal{O}_{K})$) implies $f(x) = x^{p^{\ell}} - x + c \in \mathcal{O}_{K}[x]\subset K[x]$ is irreducible modulo fixed prime $p\mathcal{O}_{K}$. So now, as in Section \ref{sec7}, we may to each irreducible polynomial $f$ associate a field $K_{f} = K[x]\slash (f(x))$, which is again a number field of deg$(f) = p^{\ell}$ over $K$; and more to this, we also recall from algebraic number theory that associated to $K_{f}$ is an integer Disc$(K_{f})$ called the discriminant. Moreover, bearing in mind that we now obtain an inclusion $\mathbb{Q}\hookrightarrow K \hookrightarrow K_{f}$ of number fields, we then also note that the degree $m:=[K_{f} : \mathbb{Q}] = [K : \mathbb{Q}] \cdot [K_{f} : K] = np^{\ell}$, for fixed degree $n\geq 2$ of any real number field $K$. So now, inspired again by field-counting advances in arithmetic statistics, we also wish to count the number of fields $K_{f}$ induced by irreducible polynomials $f$ arising from a polynomial discrete dynamical system in Section \ref{sec2}. To do so, we define and then determine the asymptotic behavior of 
\begin{equation}\label{N_{m}}
N_{m}(X) := \# \bigg\{K_{f}\slash \mathbb{Q} : [K_{f} : \mathbb{Q}] = m \textnormal{ and } |\text{Disc}(K_{f})|\leq X \bigg\}
\end{equation} as a positive real number $X\to \infty$. With that in mind, motivated greatly by recent work of Lemke Oliver-Thorne \cite{lem} and then applying the first part of their [\cite{lem}, Theorem 1.2] on the function $N_{m}(X)$, we then obtain here:

\begin{cor} \label{8.5}Fix any real algebraic number field $K\slash \mathbb{Q}$ of degree $n\geq 2$ with the ring of integers $\mathcal{O}_{K}$. Assume Corollary \ref{7.1} or second part of Theorem \ref{2.3}, and let $N_{m}(X)$ be the number defined as in \textnormal{(\ref{N_{m}})}. Then we have 
\begin{equation}\label{N_{m}(x)} 
N_{m}(X)\ll_{m}X^{2d - \frac{d(d-1)(d+4)}{6m}}\ll X^{\frac{8\sqrt{m}}{3}}, \text{where d is the least integer for which } \binom{d+2}{2}\geq 2m + 1.
\end{equation}
\end{cor}

\begin{proof}
To see the inequality \textnormal{(\ref{N_{m}(x)})}, we first recall from Corollary \ref{7.1} the existence of infinitely many monic polynomials $f(x)$ over $\mathbb{Q}\subset K_{f}$ such that $K_{f}\slash \mathbb{Q}$ is an algebraic number field of degree $m=np^{\ell}$, or recall from the second part of Theorem \ref{2.3} the existence of monic integral polynomials $f(x) = x^{p^{\ell}}-x + c\in K[x]$ that are irreducible modulo fixed prime ideal $p\mathcal{O}_{K}$ for every coefficient $c\not \in p\mathcal{O}_{K}$ and so induce degree-$m$ number fields $K_{f}\slash \mathbb{Q}$. This then means that the set of number fields $K_{f}\slash \mathbb{Q}$ of degree $m$ is not empty. So now, we may then apply [\cite{lem}, Theorem 1.2 (1)] on the number $N_{m}(X)$, and doing so we then obtain inequality \textnormal{(\ref{N_{m}(x)})}, as needed.
\end{proof}

Similarly, we also recall that we proved in Corollary \ref{7.2} the existence of an infinite family of irreducible monic integer polynomials $g(x) = x^{(p-1)^{\ell}} - x + c \in \mathbb{Q}[x]\subset K[x]$ for any fixed $\ell \in \mathbb{Z}_{\geq 1}$; and more to this, we again recall that the second part of Theorem \ref{3.3} (i.e., the part in which we proved $M_{c}(p) = 0$ for every $c\equiv -1\ (\text{mod} \ p\mathcal{O}_{K})$) implies $g(x) = x^{(p-1)^{\ell}} - x + c \in \mathcal{O}_{K}[x]\subset K[x]$ is irreducible modulo  fixed prime $p\mathcal{O}_{K}$. As before, we may to each irreducible polynomial $g$ associate a number field $L_{g} = K[x]\slash (g(x))$ of even deg$(g) = (p-1)^{\ell}$ over $K$; and also recall that attached to $L_{g}$ is an integer Disc$(L_{g})$. Since we also  obtain $\mathbb{Q}\hookrightarrow K \hookrightarrow L_{g}$ of fields, we then also note $r:=[L_{g} : \mathbb{Q}] = [K : \mathbb{Q}] \cdot [L_{g} : K] = n(p-1)^{\ell}$ for any fixed $n$. So now, we also wish to count the number of fields $L_{g}$ induced by irreducible polynomials $g$ arising from a polynomial discrete dynamical system in Section \ref{sec3}; and which we also again do by determining the asymptotic behavior of  
\begin{equation}\label{M_{r}}
M_{r}(X) := \# \bigg\{L_{g}\slash \mathbb{Q} : [L_{g} : \mathbb{Q}] = r \textnormal{ and} \ |\text{Disc}(L_{g})|\leq X \bigg\}
\end{equation} as a positive real number $X\to \infty$. By again taking great advantage of [\cite{lem}, Theorem 1.2 (1)], we then obtain:

\begin{cor} Fix any algebraic number field $K\slash \mathbb{Q}$ of degree $n\geq 2$ with the ring of integers $\mathcal{O}_{K}$. Assume Corollary \ref{7.2} or second part of Theorem \ref{3.3}, and let $M_{r}(X)$ be the number defined as in \textnormal{(\ref{M_{r}})}. Then we have 
\begin{equation}\label{M_{r}(x)}
M_{r}(X)\ll_{r}X^{2d - \frac{d(d-1)(d+4)}{6r}}\ll X^{\frac{8\sqrt{r}}{3}}, \text{where d is the least integer for which } \binom{d+2}{2}\geq 2r + 1.
\end{equation}
\end{cor}

\begin{proof}
By applying a similar argument as in the Proof of Cor. \ref{8.5}, we then obtain inequality \textnormal{(\ref{M_{r}(x)})}, as needed.
\end{proof}

By applying again a result of Bhargava-Shankar-Wang [\cite{sch1}, Corollary 1.3] on the algebraic number fields $K_{f}\slash \mathbb{Q}$ induced by irreducible monic polynomials $f$ (defined over any fixed real algebraic number field $K$ and hence over $\mathbb{Q}$) arising from a polynomial discrete dynamical system in Sect.\ref{sec2}, we then also obtain the following:
\begin{cor}\label{8.7}
Assume Corollary \ref{7.1} or second part of Theorem \ref{2.3}. The number of isomorphism classes of number fields $K_{f}\slash \mathbb{Q}$ of degree $m$ and $|\Delta(K_{f})| < X$ that are monogenic and have Galois group $S_{m}$ is $\gg X^{\frac{1}{2} + \frac{1}{m}}$.
\end{cor}

\begin{proof}
To see this, we first recall from Corollary \ref{7.1} the existence of infinitely many monic polynomials $f(x)$ over $\mathbb{Q}\subset K_{f}$ such that $K_{f}\slash \mathbb{Q}$ is an algebraic  number field of degree $m=np^{\ell}$, or recall from the second part of Theorem \ref{2.3} the existence of monic integral polynomials $f(x) = x^{p^{\ell}}-x + c\in K[x]$ that are irreducible modulo any fixed prime ideal $p\mathcal{O}_{K}$ for every coefficient $c\not \in p\mathcal{O}_{K}$ and so induce degree-$m$ number fields $K_{f}\slash \mathbb{Q}$. This then means that the set of degree-$m$ number fields $K_{f}\slash \mathbb{Q}$ is not empty. So now, applying [\cite{sch1}, Corollary 1.3] to the underlying number fields $K_{f}$ with $|\Delta(K_{f})| < X$ that are monogenic and have associated Galois group $S_{m}$, we then obtain that the number of isomorphism classes of such number fields $K_{f}$ is $\gg X^{\frac{1}{2} + \frac{1}{m}}$, as needed.
\end{proof}

As before, applying (BSW) [\cite{sch1}, Corollary 1.3] on the algebraic number fields $L_{g}\slash \mathbb{Q}$ induced by irreducible monic polynomials $g$ arising from a polynomial discrete dynamical system in Sect.\ref{sec3}, we then also obtain here:
\begin{cor}
Assume Corollary \ref{7.2} or second part of Theorem \ref{3.3}. Then the number of isomorphism classes of number fields $L_{g}\slash \mathbb{Q}$ of degree $r$ and $|\Delta(L_{g})| < X$ that are monogenic and have Galois group $S_{r}$ is $\gg X^{\frac{1}{2} + \frac{1}{r}}$.
\end{cor}

\begin{proof}
By applying a similar argument as in the Proof of Corollary \ref{8.7}, we then obtain the count, as needed.
\end{proof}

\addcontentsline{toc}{section}{Acknowledgements}
\section*{\textbf{Acknowledgements}}
I’m deeply indebted to my long-time great advisors, Dr. Ilia Binder and Dr. Arul Shankar, for all their boundless generosity, friendship and for all the inspiring weekly conversations and along with Dr. Jacob Tsimerman for always very uncomprimisingly supporting my professional and philosophical-mathematical research endeavours. I’m truly very grateful and indebted to Dr. Florian Herzig, who very wholeheartedly rescheduled his first weekly lecture of Algebra MAT1101 in the winter 2024, so that, me and some other few students would also attend the whole 2Hrs of the first weekly lecture of Algebraic Number theory MAT1200 course which at that time was taught by Dr. Tsimerman (and thank you very much to Dr. Tsimerman for being very understanding and for adjusting his teaching pace possibly as a way to accommodate some of us that had a lecture conflict with his amazing class and that of Dr. Herzig). I also more importantly, thank Dr. Herzig for the enlightening conversation on mathematical research and for his invaluable suggestions on how I could go about improving further my mathematical practice. Last but not least, I’m truly very grateful to the Dept. of MCS for everything. As a graduate research student, this work and my studies are hugely and wholeheartedly funded by Dr. Binder and Dr. Shankar. This article is dedicated to teachers who help to spark a true spirit of learning and discovering in every child in the world! Any opinions expressed in this article belong solely to me, the author, Brian Kintu; and should never be taken as a reflection of the views of anyone that’s been happily acknowledged by the author. 

\bibliography{References}
\bibliographystyle{plain}

\noindent Dept. of Math. and Comp. Sciences (MCS), University of Toronto, Mississauga, Canada \newline
\textit{E-mail address:} \textbf{brian.kintu@mail.utoronto.ca}\newline 
\date{\small{\textit{January 15, 2026}}}

\end{document}